\documentclass[preprint,onecolumn]{IEEEtran}
\def\GC{\mathcal{C}_{\dm,p}}

\def\be{\begin{equation}}
\def\ee{\end{equation}}
\def\beaN{\setlength{\arraycolsep}{0.0em}\begin{eqnarray*}}
\def\eeaN{\end{eqnarray*}\setlength{\arraycolsep}{5pt}}
\def\bea{\setlength{\arraycolsep}{0.0em}\begin{eqnarray}}
\def\eea{\end{eqnarray}\setlength{\arraycolsep}{5pt}}

\def\dm{n}

\def\Nd{\NN^\dm}
\def\Rd{\RR^\dm}
\def\Zd{\ZZ^\dm}

\def\Td{\TT^\dm}

\def\NN{\mathbb{N}}
\def\RR{\mathbb{R}}
\def\ZZ{\mathbb{Z}}
\def\TT{\mathbb{T}}
\def\Fp{F_p}

\def\dil{\Lam}

\def\set{\Gamma}
\def\mod{{\rm mod}\,}

\newtheorem{theorem}{Theorem}
\newtheorem{corollary}{Corollary}
\newtheorem{lemma}{Lemma}

\newtheorem{result}{Result}
\newtheorem{definition}{Definition}
\newtheorem{remark}{Remark}
\newtheorem{example}{Example}
\newtheorem{proof}{Proof}

\def\beaN{\setlength{\arraycolsep}{0.0em}\begin{eqnarray*}}
\def\eeaN{\end{eqnarray*}\setlength{\arraycolsep}{5pt}}
\def\bea{\setlength{\arraycolsep}{0.0em}\begin{eqnarray}}
\def\eea{\end{eqnarray}\setlength{\arraycolsep}{5pt}}

\def\be{\begin{equation}}
\def\ee{\end{equation}}

\def\dm{n}

\def\Nd{\NN^\dm}
\def\Rd{\RR^\dm}
\def\Zd{\ZZ^\dm}

\def\Td{\TT^\dm}

\def\NN{\mathbb{N}}
\def\RR{\mathbb{R}}
\def\ZZ{\mathbb{Z}}
\def\TT{\mathbb{T}}

\def\du{{\rm d}}
\def\disp{\displaystyle}

\def\bks{\backslash}

\def\gam{\gamma}

               \def\Lam{\Lambda}

\def\ome{\omega}                

\def\hata{{\widehat a}}

\def\hatg{{\widehat g}}
\def\hatG{{\widehat G}}
\def\hath{{\widehat h}}
\def\hatH{{\widehat H}}

\def\hat#1{{\widehat {#1}}}

\usepackage{cite}
\usepackage{graphicx}
\usepackage{amssymb}
\usepackage{enumerate}
\usepackage{bm,bbm}
\usepackage{subfigure}
\usepackage{booktabs}
\usepackage{tabularx}
\usepackage{array}
\usepackage{tikz}

\begin{document}




\title{Prime Coset Sum: A Systematic Method for Designing Multi-D Wavelet Filter Banks with Fast Algorithms}

\author{Youngmi~Hur$^{1,2}$ and Fang~Zheng$^{1}$\thanks{This research was partially supported by NSF Grant DMS-1115870.}
\\

$^1${\it Department of Applied Mathematics and Statistics, Johns Hopkins University, Baltimore, MD 21218, USA} \\
$^2${\it Department of Mathematics, Yonsei University, Seoul 120-749, Korea}}


\maketitle


\begin{abstract}
As constructing multi-D wavelets remains a challenging problem, we propose a new method called prime coset sum to construct multi-D wavelets. Our method provides a systematic way to construct multi-D non-separable wavelet filter banks from two 1-D lowpass filters, with one of whom being interpolatory. Our method has many important features including the following: 1) it works for any spatial dimension, and any prime scalar dilation, 2) the vanishing moments of the multi-D wavelet filter banks are guaranteed by certain properties of the initial 1-D lowpass filters, and furthermore, 3) the resulting multi-D wavelet filter banks are associated with fast algorithms that are faster than the existing fast tensor product algorithms.
\end{abstract}





\section{Preliminaries}
\label{section:preliminaries}
\subsection{Introduction}
\label{subS:introduction}
Wavelet representation has been one of the most popular data representations in the last two decades. Wavelet filter banks, which can lead to wavelet systems in $L_2(\Rd)$ under some well-understood constraints, has been widely used in Signal Processing applications. In order to obtain wavelet representation for multi-dimensional (multi-D) data, one needs multi-D wavelets. Tensor product is the most common method for constructing multi-D wavelets, and the resulting wavelets are typically referred to as the separable wavelets. However, the separable wavelets constitute only a small portion of multi-D wavelets, and they have some unavoidable limitations. One of the limitations of tensor-product-based wavelets is that the resulting multi-D filters have dense supports. It is well known that the fast algorithms associated with tensor-product-based wavelets have a complexity constant (cf. Section~\ref{subS:algorithms} for the definition of complexity constant) that increases linearly with the spatial dimension $\dm$. While this complexity may be satisfactory for many signal processing applications, it can pose a problem for many other signal processing applications, including the case when we deal with large volume data such as medical images in \cite{R}, Geographic Information Systems images in \cite{Uy} and seismic data in \cite{VED}. Moreover, it is known that tensor-product-based discrete wavelet transform is memory consuming and cannot directly obtain the target subband signals, due to its dependent subband decomposition process \cite{HGC}. There have been many researches on improving the implementation of the existing tensor-product-based wavelets \cite{OM,MW,PN,DCL}, as well as on constructing new non-tensor-based multi-D wavelets 
\cite{KV1,CS,KV2,AVM,MX,RiS1,DKT,HeL,JRS,C,CaDo,KoSw,PM,Ve,SCD,SS,CMPX,CMX,Tr,DV3,HYY,QLCCR,YDCC,Z}. However, most of these new constructions work only for low dimensions or have additional constraints on the lowpass filters. Furthermore, most of them are not associated with fast algorithms, preventing them from being widely used in practice.

Recently, the authors introduced a new method called coset sum for constructing non-tensor-based multi-D wavelets in \cite{HZ}. There it was shown that the resulting wavelets are associated with fast algorithms whose complexity constant does not increase as the spatial dimension increases. It was also shown there that many features of tensor product that makes it attractive in wavelet construction still hold true for coset sum. 

However, similar to the tensor product method, coset sum also assumes the dyadic dilation. We recall that the $\dm\times\dm$ matrix $\dil$ is called a {\it dilation matrix} if it is an integer matrix whose spectrum lies outside the closed unit disc. It determines the exact way of how downsampling and upsampling are performed in wavelets or wavelet filter banks. The dilation is called {\it scalar} if the dilation matrix is a scalar multiple of the identity matrix ${\tt I}_\dm$, i.e., $\dil=\lambda  {\tt I}_\dm$ with $\lambda\ge2$ an integer. In particular, it is called {\it dyadic} if $\dil=2 {\tt I}_\dm$. In this paper, we say that the dilation is {\it prime} if $\dil=p{\tt I}_\dm$ for a prime number $p$. Wavelets with dyadic dilation are referred to as dyadic wavelets. Dyadic wavelets are the standard and traditional types of wavelets, however they are not suitable for all applications (see, for example, \cite{PoCa,XXW,GLWM}).  

In this paper, we show that we can generalize the coset sum in the sense that multi-D wavelet filter banks with fast algorithms can be constructed for any prime dilation $p {\tt I}_\dm$. We also show that the complexity constant for our fast algorithms with prime dilation $p {\tt I}_\dm$ is independent of the spatial dimension.

The organization of this paper is as follows. The rest of Section~\ref{section:preliminaries} is a brief review of some relevant concepts including the coset sum method. In Section~\ref{section:pcs} we discuss a possible generalization of the coset sum, which we call prime coset sum, together with its properties. In Section~\ref{section:FBs} we present a new method to construct multi-D wavelet filter banks based on the prime coset sum refinement masks and show that they are associated with fast algorithms. Section~\ref{section:conclusion} is a summary of our results. Some technical proofs and details in this paper are placed in Appendix. 

\subsection{Notation and Basic Concepts}
\label{subS:notation}
Let $\dil$ be a dilation matrix and let $q:=|\det\dil|$. In the multiresolution analysis \cite{Ma1} setting, the (compactly supported) scaling or refinable function $\phi$ (with dilation $\dil$) satisfies the following refinement relation:
\bea
\label{eq:scaling}
\phi(\cdot)=\sum_{k\in\Zd}h_{\phi}(k)\phi(\dil\cdot-k),
\eea
where $h_{\phi}:\Zd\rightarrow\RR$ is the associated finitely supported filter with dilation $\dil$.

A mask {\it associated with} a finitely supported filter $h:\Zd\rightarrow\RR$ is a Laurent trigonometric polynomial defined as
\beaN
\tau(\ome):={1\over q}\sum_{k\in\Zd}h(k)e^{-ik\cdot\ome}=:\hath(\ome),
\eeaN
for any $\ome\in\Td:=[-\pi,\pi]^\dm$. That is, $\tau=\hath$ is the  Fourier transform of the filter $h$, up to a normalization. Throughout this paper, we use $\hata$ to denote this Fourier transform of $a$.

By taking the Fourier transform of (\ref{eq:scaling}), the refinement relation can be recast as 
\beaN
\hat{\phi}(\dil^\ast\ome)=\tau(\ome)\hat{\phi}(\ome),\quad \forall \ome\in\Td,
\eeaN
where $\tau$ is the mask associated with $h_\phi$, and the superscript $\ast$ is used to denote the conjugate transpose of a matrix, and hence $\dil^\ast$ is the same as $\dil^T$, the transpose of $\dil$, in this case.

A mask $\tau$ with $\tau(0)=0$ is typically referred to as a {\it wavelet mask}. In this paper, we use the normalization of the mask so that a mask with $\tau(0)=1$ is referred to as a {\it refinement mask}. This is equivalent to $\sum_{k\in
\Zd}h(k)=q$, which is our normalization for a filter to be {\it lowpass}. A refinement mask $\tau$ is called {\it interpolatory} if, for any $\ome\in\Td$, 
\beaN
\sum_{\gam\in\set^\ast}\tau(\ome+\gam)=1,
\eeaN
where $\set^\ast$ is a complete set of representatives of the distinct cosets of $2\pi(((\dil^\ast)^{-1}\Zd)/\Zd)$ containing $0$. For example, for the scalar dilation with $\lambda$, the set ${2\pi\over \lambda}\{0,1,\cdots,\lambda-1\}^\dm$ can be used for $\set^\ast$. We note that $\tau$ is interpolatory if and only if its corresponding filter $h$ satisfies
\bea
\label{eq:interpolatory}
h(k)=
\cases{
1,\quad\mbox{if $k=0$},\cr
            0,\quad\mbox{if $k\in\dil\Zd\bks0$}.\cr
}
\eea

The order of zeros of $\tau$ at $\gam\in\set^\ast\bks0$ is called the {\it accuracy number} of $\tau$. Throughout this paper, we assume that all refinement masks have at least accuracy number one. The order of zeros of $\tau$ at the origin is called the {\it number of vanishing moments} of $\tau$. Thus a mask is a wavelet mask if and only if it has at least one vanishing moment. The order of zeros of $1-\tau$ at the origin is called the {\it flatness number} of $\tau$. Thus a mask is a refinement mask if and only if it has at least flatness number one. Throughout this paper, we use the accuracy number, the number of vanishing moments, and the flatness number both for a mask and for the filter associated with it. 

Two refinement masks $\tau$ and $\tau^\du$ are called {\it biorthogonal} if
\beaN
\sum_{\gam\in\set^\ast}(\overline{\tau}\tau^\du)(\ome+\gam)=1,
\eeaN
for any $\ome\in\Td$. Here and below, the overline is used to denote the complex conjugate. For the corresponding filters $h$ and $g$ of $\tau$ and $\tau^\du$, respectively, the biorthogonality condition becomes 
\beaN
\sum_{k\in\Zd}h(k)g(k+\dil l)=q\delta_{l,0}=
\cases{ 
              q,\quad\mbox{if $l=0$},\cr
              0,\quad\mbox{if $l\in\Zd\bks0$}.\cr
}
\eeaN
For a pair of biorthogonal refinement masks $\tau$ and $\tau^\du$ and wavelet masks $t_j$ and $t_j^\du$, $j=1,\ldots,q-1$, we refer to $(\tau,(t_j)_{j=1,\ldots,q-1})$ and $(\tau^\du,(t_j^\du)_{j=1,\ldots,q-1})$ as the {\it combined biorthogonal masks} if they satisfy the following condition: for every $\ome\in\Td$,
\be
\label{eq:muep}
\overline{\tau(\ome+\gam)}\tau^\du(\ome)+\sum_{j=1}^{q-1}\overline{t_j(\ome+\gam)}t_j^\du(\ome)=\delta_{\gam,0}=
\cases{
1,\quad\mbox{if $\gam=0$},\cr
             0,\quad\mbox{if $\gam\in\set^\ast\bks0$}.\cr
}
\ee
It is well known that the combined biorthogonal masks can give rise to a biorthogonal wavelet system in $L_2(\Rd)$ (see, for example, \cite{RS2}). 

A filter bank is a finite set of filters. We consider only the filter banks that are non-redundant with the perfect reconstruction property \cite{SN}. A (non-redundant) filter bank  consists of analysis bank and synthesis bank, which are collections of $q=|\det\dil|$ filters linked by downsampling and upsampling operators, respectively, associated with the dilation matrix $\dil$. The analysis bank splits the input signal into $q$ signals typically called subband signals using a parallel set of bandpass filters. The synthesis bank reconstructs the original data from $q$ subband signals. We are interested in the {\it wavelet filter bank} for which each of analysis and synthesis banks has exactly one lowpass filter and the rest of them are all highpass filters. We recall that a filter $h$ is {\it highpass} if the associated mask is a wavelet mask, i.e. $\sum_{k\in\Zd}h(k)=0$. The filters associated with the combined biorthogonal masks constitute a wavelet filter bank. Furthermore, it is well known that the minimum of accuracy numbers of lowpass filters in a given wavelet filter bank provides a lower bound for the number of vanishing moments of the highpass filters in the given wavelet filter bank \cite{CHR}.

\subsection{Multi-D Wavelet Construction Methods: Tensor Product and Coset Sum}
\label{subS:CosetSum}

When $q=|\det\dil|$ is large, in general, it is not easy to find the combined biorthogonal masks $(\tau,(t_j)_{j=1,\ldots,q-1})$ and $(\tau^\du,(t_j^\du)_{j=1,\ldots,q-1})$. However, if the dilation is dyadic (i.e. $\dil=2{\tt I}_\dm$ and $q=2^\dm$) and the spatial dimension $\dm$ satisfies $n\ge2$, then the well-known tensor product and more recent coset sum can be used. Below we provide a brief review of these methods. 

We recall that the $\dm$-D tensor product mask from $\dm$ (possibly distinct) $1$-D masks $R_1,R_2,\ldots,R_\dm$ is defined as, for $\ome=(\ome_1,\ome_2,\ldots,\ome_\dm)\in\Td$,
\beaN
\mathcal{T}_\dm[R_1,R_2,\ldots,R_\dm](\ome):=R_1(\ome_1)R_2(\ome_2)\cdots R_\dm(\ome_\dm).
\eeaN
Then starting from $1$-D combined biorthogonal masks $(S_0,S_1)$ and $(U_0,U_1)$ with dyadic dilation, one can construct $\dm$-D combined biorthogonal masks with dyadic dilation by setting the $\dm$-D biorthogonal refinement masks as
\beaN
\tau:=\mathcal{T}_\dm[S_0,S_0,\ldots,S_0],\quad \tau^\du:=\mathcal{T}_\dm[U_0,U_0,\ldots,U_0],
\eeaN
and the $\dm$-D wavelet masks $t_\nu$, $t_\nu^\du$, $\nu=(\nu_1,\nu_2,\ldots,\nu_\dm)\in\{0,1\}^\dm\bks0$, as
\beaN
t_\nu=\mathcal{T}_\dm[S_{\nu_1},S_{\nu_2},\ldots,S_{\nu_\dm}],\quad t_\nu^\du=\mathcal{T}_\dm[U_{\nu_1},U_{\nu_2},\ldots,U_{\nu_\dm}].
\eeaN

It is well known that the above tensor product method has many advantages: 1) it preserves the interpolatory property and the accuracy number of $1$-D refinement masks; 2) it also preserves the biorthogonality between two refinement masks; and 3) the resulting separable wavelets are associated with fast algorithms (cf. Section~\ref{subS:algorithms}). However, as discussed in Section \ref{subS:introduction}, the limitations of the separable wavelets constructed from the tensor product are widely known. 

Aa an alternative to the tensor product, a new method called coset sum for constructing $\dm$-D dyadic refinement masks from $1$-D dyadic refinement masks is recently proposed \cite{HZ}. The coset sum refinement mask $\mathcal{C}_\dm[R]$ for a $1$-D dyadic refinement mask $R$ is defined as
\beaN
\mathcal{C}_\dm[R](\ome):={1\over
2^{\dm-1}}\left(1-2^{\dm-1}
+\sum_{\nu\in\{0,1\}^{\dm}\bks0}R(\ome\cdot\nu)\right),\quad\ome\in\Td.
\eeaN
The following results about coset sum refinement masks and coset sum wavelet systems have been proved in \cite{HZ}.

\begin{result}
\label{result:result1}
Let $\mathcal{C}_\dm$ be the coset sum, and let $R$ and $\tilde{R}$ be univariate dyadic refinement masks.
\begin{enumerate}[(a)]
\item $\mathcal{C}_\dm[R]$ is interpolatory if and only if $R$ is interpolatory.
\item Suppose that one of $R$ and $\tilde{R}$ is interpolatory. Then $\mathcal{C}_\dm[R]$ and $\mathcal{C}_\dm[\tilde{R}]$ are biorthogonal if and only if $R$ and $\tilde{R}$ are biorthogonal.
\item Suppose that $R$ is interpolatory. Then $\mathcal{C}_\dm[R]$ and $R$ have the same accuracy number.\qquad\endproof
\end{enumerate}
\end{result}

\begin{result}
\label{result:result2}
Suppose that $S$ and $U$ are 1-D biorthogonal dyadic refinement masks, and that $U$ is interpolatory. Define $\dm$-D biorthogonal refinement masks as
\beaN
\tau:=\mathcal{C}_\dm[S],\quad
\tau^\du:=\mathcal{C}_\dm[U],
\eeaN
and $\dm$-D wavelet masks $t_\nu$, $\nu\in\{0,1\}^{\dm}\bks0$, as
\be
\label{eq:nDwaveletmask}
t_\nu(\ome)=e^{-i\ome\cdot\nu}\overline{U(\ome\cdot\nu+\pi)}, \quad \ome\in\Td.
\ee
Then there exist wavelet masks $t_\nu^\du$, $\nu\in\{0,1\}^{\dm}\bks0$, such that $(\tau,(t_\nu)_{\nu\in\{0,1\}^{\dm}\bks0})$ and $(\tau^\du,(t_\nu^\du)_{\nu\in\{0,1\}^{\dm}\bks0})$ are $\dm$-D combined biorthogonal masks with dyadic dilation.
\qquad\endproof
\end{result}

\medskip
As we can see above, the coset sum and the tensor product method share many useful properties. In addition, the coset sum wavelets can overcome some of the limitations of the separable wavelets. For example, attributed to the smaller {\it supports} (number of nonzero entries) of the resulting multi-D filters, as well as the special structure of the filters, the coset sum can be associated with fast algorithms whose complexity constant does not increase with the spatial dimension. Therefore, in higher dimension, coset sum fast algorithms can be much faster than the tensor product fast algorithms. For more details about the coset sum including its comparison with the tensor product, we refer to \cite{HZ}.

\section{Prime Coset Sum}
\label{section:pcs}

Since coset sum has many useful properties including fast algorithms, which can be much faster than the existing tensor product fast algorithms, in this section, we try to extend the coset sum method to non-dyadic scalar dilations. The following simple lemma plays an important role in our generalization of coset sum.

\begin{lemma}
\label{lemma:prime}
Let $\dm\ge 1$ be a fixed spatial dimension. 
Let $p$ be a prime number, and let $\set$ and $\set^\ast$ be the complete set of representatives of the distinct cosets of $\Zd/p\Zd$ and $2\pi((p^{-1}\Zd)/\Zd)$, respectively, containing $0$. Then for every $\gam\in\set^{\ast}\bks0$, we have
\beaN
\#\{\nu\in\set: \;  \gam\cdot\nu\equiv0\,(\mod 2\pi\ZZ) \}=p^{\dm-1}.
\eeaN
\end{lemma}

\begin{remark}
A special case of Lemma~\ref{lemma:prime} for $p=2$ is used for the coset sum (cf. (19) in \cite{HZ}).
\qquad\endproof
\end{remark}

\begin{remark}
\label{remark:prime}
In general, Lemma~\ref{lemma:prime} does not hold true if $p$ is not a prime number. For example, when $p=4$ and $\dm=1$, we can take $\set=\{0,1,2,3\}$ and $\set^{\ast}\bks0=\{{2\pi\over4},{4\pi\over4},{6\pi\over4}\}$. Then, it is easy to see that if $\gam={2\pi\over4}$ or $\gam={6\pi\over4}$, then the cardinality of the set $Z_\gam:=\{\nu\in\set: \;  \gam\cdot\nu\equiv0\,(\mod 2\pi\ZZ) \}$ is 1 (in fact, $Z_\gam=\{0\}$ in both cases), whereas if $\gam={4\pi\over4}$, then $Z_\gam=\{0,2\}$ and hence its cardinality is $2$. As we will see below, in our proof of the lemma, we used crucially the fact that $\ZZ/p\ZZ$ is a finite field for a prime number $p$, which does not hold true anymore if $p$ is not a prime number. 
\qquad\endproof
\end{remark}

\begin{proof}[Proof of Lemma~\ref{lemma:prime}]
First of all, we claim that, without lose of generality, we may assume $\set=\{0,1,\cdots,p-1\}^\dm$ and $\set^\ast={2\pi\over p}\{0,1,\cdots,p-1\}^\dm$. This is because 
for any other $\tilde{\set}$ and $\tilde{\set}^\ast$, there is a one-to-one correspondence between the elements of $\tilde{\set}$ and $\set$, and between the elements of $\tilde{\set}^\ast$ and $\set^\ast$. To be more specific, 
for any other $\tilde{\set}$ and $\tilde{\set}^\ast$, and for any $\tilde{\nu}\in\tilde{\set}$ and $\tilde{\gam}\in\tilde{\set}^\ast\bks 0$, there exist unique $\nu\in\set$ and $\gam\in\set^\ast\bks 0$ such that 
\beaN
\nu\equiv\tilde{\nu} \,(\mod p\Zd), \quad {p\over2\pi}\gam\equiv{p\over2\pi}\tilde{\gam} \, (\mod p\Zd),
\eeaN
and vice versa. Therefore, $\tilde{\gam}\cdot\tilde{\nu} \equiv \gam\cdot\nu \, (\mod 2\pi\ZZ)$. Hence the cardinality of the set $\{\nu\in\set: \;  \gam\cdot\nu\equiv0\,(\mod 2\pi\ZZ) \}$ is the same as the cardinality of the set $\{\tilde{\nu}\in\tilde{\set}: \;  \tilde{\gam}\cdot\tilde{\nu}\equiv0\,(\mod 2\pi\ZZ) \}$.

Now for any $\gam\in\set^\ast\bks0={2\pi\over p}\{0,1,\cdots,p-1\}^\dm\bks0$, and $\nu\in\set=\{0,1,\cdots,p-1\}^\dm$, we let $\mu:={p\over2\pi}\gam$, and let $\mu_i$ and $\nu_i$, $i=1,\ldots,\dm$, be the $i$-th component of $\mu$ and $\nu$. Then both $\mu_i$ and $\nu_i$ lie in the set $\{0,1,\cdots,p-1\}$. Since $\gam\neq0$, at least one of $\mu_i$'s is not 0.  Without loss of generality, we may assume $\mu_\dm\neq0$. Furthermore, $\gam\cdot\nu\equiv0 \,(\mod 2\pi\ZZ)$ if and only if
$\mu_1\nu_1+\cdots+\mu_{\dm}\nu_{\dm}\equiv 0 \,(\mod p\ZZ)$. 

For any $\gam\in\set^\ast\bks0$, and any $\nu_i\in\{0,1,\cdots,p-1\}$, $i=1,\ldots,\dm-1$, let $k\in\{0,1,\cdots,p-1\}$ satisfy 
\beaN
\mu_1\nu_1+\cdots+\mu_{\dm-1}\nu_{\dm-1}\equiv k \,(\mod p\ZZ).
\eeaN
Since $\ZZ/p\ZZ$ is a finite field for a prime number $p$, there exists a unique multiplicative inverse $\rho(\mu_\dm)\in\{1,\cdots,p-1\}$ of $\mu_\dm$ such that $\mu_\dm \rho(\mu_\dm)\equiv1 \,(\mod p\ZZ)$. Then there exists a unique $\nu_\dm\in\{0,1,\cdots,p-1\}$ satisfies
 \beaN
 \nu_\dm\equiv(-k)\rho(\mu_\dm) \,(\mod p\ZZ).
 \eeaN
 Thus 
 \beaN
 \mu_1\nu_1+\cdots+\mu_{\dm-1}\nu_{\dm-1}+\mu_\dm\nu_\dm
 \equiv k+\mu_\dm(-k)\rho(\mu_\dm)
 \equiv 0 \, (\mod p\ZZ).
 \eeaN
 Since there are $p^{\dm-1}$ different choices for $\nu_1,\nu_2,\cdots,\nu_{\dm-1}$, for any $\gam\in\set^\ast\bks0$, we have
\beaN
\#\{\nu\in\set: \;  \gam\cdot\nu\equiv0 \,(\mod 2\pi\ZZ) \}=p^{\dm-1}.
\eeaN
\end{proof}

\medskip
With Lemma~\ref{lemma:prime} in hand,  we define a particular generalization of coset sum for the prime dilation $\dil=p{\tt I}_\dm$, where $p\ge 2$ is a prime number. Let $\set$ and $\set^\ast$ be defined as in Lemma~\ref{lemma:prime}. For example, $\set=\{0,1,\cdots,p-1\}^{\dm}$ and $\set^\ast={2\pi\over p}\{0,1,\cdots,p-1\}^\dm$ can be used. 

Motivated by the definition of the original coset sum $\mathcal{C}_\dm$ (cf. Section~\ref{subS:CosetSum}), we consider a generalized coset sum $\GC$ of the form
\beaN
\GC[R](\ome)=A\left(B
+\sum_{\nu\in\set'}R(\ome\cdot\nu)\right),
\eeaN
where $\set':=\set\bks0$, and $A$ and $B$ are constants that will be determined soon. 
To pin down the constants $A$ and $B$, we impose two conditions that we consider {\it natural} on the map $\GC$. Firstly, we require $\GC$ to map a 1-D refinement mask with dilation $p$ to an $\dm$-D refinement mask with dilation $p{\tt I}_\dm$. That is, we want $\GC[R](0)=1$ whenever $R(0)=1$. From this we get the equation
\be
\label{eqn:1}
B+p^\dm-1={1\over A}.
\ee
Secondly, we require the accuracy number of $\GC[R]$ to be at least one whenever the accuracy number of the 1-D refinement mask $R$ is at least one. That is, we want, for any $\gam\in\set^\ast\bks0$, 
\beaN
0=\GC[R](\gam)=A\Bigg(B+\sum_{\{\nu\in\set', \gam\cdot\nu\equiv0\}}R(0)\Bigg)=A\Bigg(B+p^{\dm-1}-1\Bigg),
\eeaN
where the last equality is due to Lemma~\ref{lemma:prime}. This gives the equation
\be
\label{eqn:2}
B+(p^{\dm-1}-1)=0.
\ee
By solving $A$ and $B$ that satisfy (\ref{eqn:1}) and (\ref{eqn:2}) simultaneously, we reach the following definition of a generalized coset sum for prime dilations.

\begin{definition}
\label{def:primecosetsum}
Let $p$ be a prime number. We define the {\it prime coset sum}
$\GC$ that maps a 1-D refinement mask
$R$ with dilation $p$ to an $\dm$-D refinement mask $\GC[R]$ with dilation $p{\tt I}_\dm$ as follows:
for any $\ome\in\Td$,
\beaN
\GC[R](\ome):={1\over
(p-1)p^{\dm-1}}\left(1-p^{\dm-1}
+\sum_{\nu\in\set'}R(\ome\cdot\nu)\right),
\eeaN
where $\set'=\set\bks0$.
\qquad\endproof
\end{definition}

\begin{remark}
We refer to the refinement mask obtained by $\GC$ as the {\it prime coset sum refinement mask}. We notice that the prime coset sum $\GC$ with $p=2$ reduces to the original coset sum $\mathcal{C}_\dm$ for dyadic dilation, i.e. $\mathcal{C}_{\dm,2}=\mathcal{C}_\dm$ (cf. Section~\ref{subS:CosetSum} for the choice of $\set=\{0,1\}^\dm$ and \cite{HZ} for more general choice of $\set$). 
\qquad\endproof
\end{remark}

\medskip
Let $H$ be the 1-D lowpass filter associated with the 1-D refinement mask $R$. Let $h$ be the $\dm$-D lowpass filter associated with the $\dm$-D refinement mask $\GC[R]$. We refer to such a filter $h$ as the {\it prime coset sum lowpass filter}. For any nonzero $k\in\Zd$, we define a set $W_k$ as $W_k:=\{l\in\ZZ\bks0: k=l\nu \hbox{ for some } \nu\in\set'\}$. Then the $\dm$-D prime coset sum lowpass filter $h$ can be written in terms of the 1-D lowpass filter $H$ as follows: 
\be
\label{eq:filter}
h(k)=
\cases{
             {\disp1\over p-1}(p-p^\dm+(p^\dm-1)H(0)),&\mbox{if $k=0$},\cr
{\disp1\over p-1}\sum_{l\in W_k}H(l),&\mbox{if $k\ne 0$}.\cr
}
\ee

Now we give a simple example to show the construction of multi-D prime coset sum lowpass filters.

\begin{example}[\bf Centered $2$-D Haar lowpass filter with dilation $3$]
\label{example:1}
Consider the centered $1$-D Haar lowpass filter with dilation $3$:
$$
H(K)=
\cases{
             1, \quad\mbox {if $K=0$ or $K=\pm 1$},\cr
             0, \quad\mbox{otherwise}.\cr
    }
$$
Let us take $\set=\{-1,0,1\}^2=\{(0,0),\pm(1,0),\pm(0,1),\pm(1,1),\pm(1,-1)\}$. Then it is easy to check that the $2$-D prime coset sum lowpass filter constructed from the $1$-D centered Haar is
$$
h(k)=
\cases{
            1, \quad\mbox {if $k=(0,0)$, $k=\pm (1,0)$, $k=\pm (0,1)$, $k=\pm (1,-1)$ or $k=\pm (-1,1)$},\cr
            0, \quad\mbox{otherwise}.\cr
    }
$$
Figure \ref{figure:haar} shows the $1$-D filter $H$ and the resulting $2$-D filter $h$. 
\qquad\endproof
\end{example}

\begin{figure}[t]
\centering
1 \quad {\bf1} \quad 1 \quad
$\longrightarrow$~~~
\begin{tabular}{ccc}
1                          & 1& 1\\[3pt]
1 &            {\bf 1}              & 1\\[3pt]
1&1&   1                     \\[3pt]
\end{tabular}
\caption[]{Construction of centered $2$-D Haar lowpass filter with dilation $3$ using prime coset sum (cf.~Example \ref{example:1}) \footnotemark}
\label{figure:haar}
\end{figure}

\footnotetext{Bold-faced number indicates that it is at the origin. This figure is also given out in \cite{HZAsilomar}. }

\medskip
Some of the properties of the original coset sum (cf. Section~\ref{subS:CosetSum}) still hold true for the generalized prime coset sum. 

\begin{lemma}
\label{lemma:interpolatory}
Let $\GC$ be the prime coset sum, and $R$ be a univariate refinement mask with dilation $p$. If $R$ is interpolatory, then $\GC[R]$ is interpolatory.
\end{lemma}

\begin{proof}
See Appendix~\ref{subS:proofoflemmainterpolatory}. 
\end{proof}

\begin{lemma}
\label{lemma:accuracy}
Let $\GC$ be the prime coset sum, $R$ be a univariate refinement mask with dilation $p$, and let $m_1$ and $m_2$ be positive integers. Suppose that $R$ has $m_1$ accuracy and $m_2$ flatness. Then $\GC[R]$ has at least $\min\{m_1,m_2\}$ accuracy.
\end{lemma}

\begin{proof}
See Appendix~\ref{subS:proofoflemmaaccuracy}. Similar arguments to the ones given in \cite{HZ} are used in our proof.
\end{proof}

\begin{remark}
If $R$ is interpolatory, then $m_1=m_2$. Hence, the above lemma says that, when $R$ is interpolatory, the accuracy number of $\GC[R]$ is at least as much as the accuracy number of $R$. For the case of the original coset sum with dyadic dilation, the accuracy number of $\mathcal{C}_\dm[R]$ is exactly the same as the accuracy number of $R$ when $R$ is interpolatory (cf. Result~\ref{result:result1}(c)). We do not yet know whether this result would hold true for the prime coset sum in general.  
\qquad\endproof
\end{remark}

\begin{lemma}
\label{lemma:flatness}
Let $\GC$ be the prime coset sum, and $R$ be a univariate refinement mask with dilation $p$. Then the flatness number of $\GC[R]$ is at least the flatness number of $R$.
\end{lemma}

We omit the proof of Lemma~\ref{lemma:flatness} as it is a simple variant of our proof of Lemma~\ref{lemma:accuracy}. 

\medskip
Unlike the original coset sum with dyadic dilation (cf. Result~\ref{result:result1}(b)), in general, the prime coset sum does not preserve the biorthogonality of 1-D refinement masks when $p>2$, even if one of them is interpolatory. Let us look at two examples to this end. Both of them are related with the Haar refinement masks with dilation $3$.

\begin{example}[\bf Centered $2$-D Haar refinement mask with dilation $3$]
\label{example:2}
Let us consider the centered 1-D Haar refinement mask as in Example \ref{example:1}: 
$$\disp{1\over 3}\left(e^{i\ome}+1+e^{-i\ome}\right).$$
Then the above mask has dilation $3$ and it is associated with the refinable function $\phi=\chi_{[-1/2,1/2]}$. If we define both $R$ and $\tilde{R}$ to be this centered $1$-D Haar refinement mask with dilation $3$, then they are interpolatory and biorthogonal with one accuracy.

Let us now take $\set=\{-1,0,1\}^2=\{(0,0),\pm(1,0),\pm(0,1),\pm(1,1),\pm(1,-1)\}$. Then, it is easy to see that transforming $R$ and $\tilde{R}$ to $2$-D using the prime coset sum with $p=3$ produces two $2$-D refinement masks $\mathcal{C}_{2,3}[R]$ and $\mathcal{C}_{2,3}[\tilde{R}]$ (cf. Figure \ref{figure:haar}) that are not only interpolatory with one accuracy, but also biorthogonal.
\qquad\endproof
\end{example}

\begin{example}[\bf Non-centered $2$-D Haar refinement mask with dilation $3$]
\label{example:3}
Now let us consider the non-centered $1$-D Haar refinement mask with dilation $3$:
\beaN 
\disp{1\over 3}\left(1+e^{-i\ome}+e^{-2i\ome}\right),
\eeaN
that is associated with the refinable function $\phi=\chi_{[0,1]}$, where $\chi_{[0,1]}$ is the characteristic function on $[0,1]$. Let both $R$ and $\tilde{R}$ be the above non-centered $1$-D Haar refinement mask with dilation $3$. Then it is easy to see that $R$ and $\tilde{R}$ are interpolatory and biorthogonal, and they have one accuracy. 

We use $\set=\{0,1,2\}^2=\{(0,0),(0,1),(0,2),(1,0),(1,1),(1,2),(2,0),(2,1),(2,2)\}$ this time. By transforming $R$ and $\tilde{R}$ to $2$-D masks using the prime coset sum with $p=3$, we see that $\mathcal{C}_{2,3}[R]$ and $\mathcal{C}_{2,3}[\tilde{R}]$ are still interpolatory and they still have one accuracy, but that they are no longer biorthogonal. 
\qquad\endproof
\end{example}

\section{Multi-D Wavelet Filter Banks with Fast Algorithms}
\label{section:FBs}
\subsection{Theory}
\label{subS:theory}

Suppose that $S$ and $U$ are 1-D biorthogonal refinement masks with dilation $p$, and that $U$ is interpolatory. Since the $\dm$-D prime coset sum refinement masks $\GC[S]$ and $\GC[U]$ are not necessarily biorthogonal (cf. Example \ref{example:3} in Section~\ref{section:pcs}), it is not trivial to construct wavelet filter banks from $\GC[S]$ and $\GC[U]$ directly. We propose to use a recent method developed by the first author \cite{Hur}. This method can construct wavelet filter banks from two refinement masks that are not necessarily biorthogonal, as long as one of them is interpolatory. 
Noting that $\GC[U]$ is interpolatory (cf. Lemma~\ref{lemma:interpolatory}), we apply this method to $\GC[S]$ and $\GC[U]$ to construct wavelet filter banks. As we will see later (cf. Section \ref{subS:algorithms}), similar to the coset sum case, the resulting wavelet filter banks using this method can be associated with fast algorithms, that are faster than the tensor product fast algorithms.

Since the method in \cite{Hur} works for any dilation matrix $\dil$, below we present it for the general dilation matrix $\dil$ with $q=|\det\dil|$. Let $\set$ and $\set^\ast$ be the complete set of representatives of the distinct cosets of $\Zd/\dil\Zd$ and $2\pi(((\dil^\ast)^{-1}\Zd)/\Zd)$, respectively, containing $0$. The following result is from \cite{Hur} written in terms of our notation. 

\begin{result}
\label{result:result3}
Suppose $g$ and $h$ are two $\dm$-D lowpass filters with dilation $\dil$, and $h$ is interpolatory. Then the two $\dm$-D refinement masks defined as
\beaN
\tau(\ome):=\hatg(\ome)+\Big(1-\sum_{\gam\in\set^\ast}\hatg(\ome+\gam)\overline{\hath(\ome+\gam)}\Big),\quad \tau^\du(\ome):=\hath(\ome),
\eeaN
for every $\ome\in\Td$, and the $\dm$-D wavelet masks defined as
\beaN
t_\nu(\ome):=e^{-i\ome\cdot\nu}-q\;\overline{(h(\nu+\dil\cdot))\widehat{\phantom{x}}(\dil^\ast\ome)},
\eeaN
and 
\beaN
t_\nu^\du(\ome):={1\over q}\;e^{-i\ome\cdot\nu}-\overline{(g(\nu+\dil\cdot))\widehat{\phantom{x}}(\dil^\ast\ome)}\;\hath(\ome),
\eeaN
for every $\ome\in\Td$, and $\nu\in\set'=\set\bks0$, form the combined biorthogonal masks (cf. (\ref{eq:muep})). 
\qquad\endproof
\end{result}

\begin{proof} 
Result~\ref{result:result3} is proved in \cite{Hur}, but under slightly different settings. For completeness, we provide an alternative proof that does not rely on the results of \cite{Hur}. Our proof is placed in Appendix~\ref{subS:proofofresult3}.
\end{proof}

\begin{remark}
\label{remark:vm}
In fact, the results in \cite{Hur} say that, if we assume that, in addition to the assumptions of Result~\ref{result:result3}, $h$ has $\alpha_1$ accuracy, $g$ has $\alpha_2$ accuracy, and $\alpha_3$ flatness, then $\tau$ has at least $\min\{\alpha_1,\alpha_2,\alpha_3\}$ accuracy. In such a case, $t_\nu$ and $t_\nu^\du$, $\nu\in\set'$, have at least $\min\{\alpha_1,\alpha_2,\alpha_3\}$ vanishing moments (cf. Section~\ref{subS:notation}). 
\qquad\endproof
\end{remark}

\medskip
For the rest of this section, we assume that the dilation is prime, i.e. $\dil=p{\tt I}_\dm$, and that the sets $\set$ and $\set^\ast$ are associated with the prime dilation, i.e., $\set$ and $\set^\ast$ are the complete set of representatives of the distinct cosets of $\Zd/p\Zd$ and $2\pi((p^{-1}\Zd)/\Zd)$, respectively, containing $0$. In particular, we have $q=|\det\dil|=p^\dm$ in this case.

Before presenting our main theorem, let us first define a map 
\beaN
\eta:\Fp' \times \set'\to\set',
\eeaN
with $\Fp':=\Fp\bks0$, where $\Fp$ is a complete set of representatives of the distinct cosets of $\ZZ/p\ZZ$ that contains 0. For example, the set $\{0,1,\cdots,p-1\}$ can be used for $\Fp$. Let $(l,\nu)\in \Fp' \times \set'\subset \ZZ\times \Zd$. Then there exists the unique multiplicative inverse $\rho(l)\in \Fp'$ of $l$ (cf. Remark \ref{remark:prime} in Section~\ref{section:pcs}). After computing the multiplication $\rho(l)\nu$ in the usual sense, we define $\eta(l,\nu)$ to be the element in $\set'=\set\bks0$ so that 
\beaN
\eta(l,\nu)\equiv \rho(l)\nu \,(\mod p\ZZ^\dm).
\eeaN 
By the above conditions, $\eta(l,\nu)$ is uniquely well defined as an element in $\set'$ since $\rho(l)\nu$ is in $\Zd$ but not in $p\Zd$. For example, if $\dm=2$, $p=3$, $\Fp=\{0,1,2\}$ and $\set=\{0,1,2\}^2$, then $\eta(2,(1,1))=(2,2)$ and $\eta(2,(2,2))=(1,1)$.

Now we are ready to present our result.
\begin{theorem}
\label{thm:wavelet}
Suppose that $G$ and $H$ are two 1-D lowpass filters with dilation $p$, and that $H$ is interpolatory. Let $S:=\hatG$ and $U:=\hatH$ be the 1-D refinement masks associated with $G$ and $H$, and let $\GC$ be the prime coset sum.  Define $\dm$-D biorthogonal refinement masks as
\beaN
\tau(\ome):=\GC[S](\ome)+\left(1-\sum_{\gam\in\set^\ast}\GC[S](\ome+\gam)\overline{\GC[U](\ome+\gam)}\right), 
\tau^\du(\ome):=\GC[U](\ome),
\eeaN
for every $\ome\in\Td$, and $\dm$-D wavelet masks as
\be
\label{eq:tnu}
t_\nu(\ome):=e^{-i\ome\cdot\nu}\left(1-{p\over p-1}\sum_{l\in \Fp'}e^{i(\ome\cdot \eta(l,\nu))l} \;\overline{U_l\Big(p\ome\cdot \eta(l,\nu)\Big)}\right),\quad \nu\in\set'
\ee
and
\bea
\label{eq:tnudu}
t_\nu^\du(\ome):={1\over p^\dm \;}e^{-i\ome\cdot\nu}\left(1-{p\over p-1} \sum_{l\in\Fp'}e^{i(\ome\cdot\eta(l,\nu)) l }\; \overline{S_l\Big(p\ome\cdot\eta(l,\nu)\Big)}\;\tau^\du(\ome) \right),
\eea
for $\nu\in\set'$, and for every $\ome\in\Td$, where $U_l(\xi):=(H(l+p\cdot))\widehat{\phantom{x}}(\xi)$, and $S_l(\xi):=(G(l+p\cdot))\widehat{\phantom{x}}(\xi)$, $\xi\in\TT$.\footnote{$U_l$ and $S_l$ can be interpreted as the polyphase decomposition of filter $H$ and  $G$, respectively (cf. Appendix~\ref{subS:ECLP}).}
Then $(\tau, (t_\nu)_{\nu\in\set'})$ and $(\tau^\du, (t_\nu^\du)_{\nu\in\set'})$ form $\dm$-D combined biorthogonal masks. 
\end{theorem}

\begin{remark}
 In the dyadic setting, i.e., when $p=2$, one can take $F_2=\{0,1\}$ and $\set=\{0,1\}^\dm$. Then, since $1$ is the only element in $F_2'$ and $\eta(1,\nu)=\nu$ for all $\nu\in\{0,1\}^\dm\bks0$, the $\dm$-D wavelet masks in (\ref{eq:tnu}) become
\beaN
t_\nu(\ome)&{\,=\,}&e^{-i\ome\cdot\nu}-2\;\overline{U_1\Big(2\ome\cdot\nu\Big)}\\
&=&e^{-i\ome\cdot\nu}-2\;\overline{e^{i\ome\cdot\nu}\Big(U(\ome\cdot\nu)-{1\over 2}\Big)}
=e^{-i\ome\cdot\nu}-e^{-i\ome\cdot\nu}\Big(1-2\;\overline{U(\ome\cdot\nu+\pi)}\Big)\\
&=&2e^{-i\ome\cdot\nu}\;\overline{U(\ome\cdot\nu+\pi)}, \quad\nu\in\{0,1\}^\dm\bks0,
\eeaN
where the second identity is from the definition of $U_1$ and the third identity is from the fact that $U$ is interpolatory. The above wavelet masks are the same as the wavelet masks in the coset sum wavelet system (cf. (\ref{eq:nDwaveletmask}) in Result~\ref{result:result2}) up to a normalization factor. In fact, the exact forms of $t_\nu^\du$ for coset sum wavelet system are also provided in \cite{HZ}, and similar calculation shows that they are the same as $t_\nu^\du$ in (\ref{eq:tnudu}) up to a normalization factor when $p=2$.
Hence we conclude that Theorem~\ref{thm:wavelet} reduces to the known result of the original coset sum case when $p=2$.
\qquad\endproof
\end{remark}

\begin{remark}
We refer to the wavelet filter bank associated with the combined biorthogonal masks constructed in Theorem~\ref{thm:wavelet} as the {\it prime coset sum wavelet filter bank}. There are many potentially useful properties of the prime coset sum wavelet filter banks. One important property is that it can be implemented by fast algorithms (cf. Section~\ref{subS:algorithms}). 
\qquad\endproof
\end{remark}

\begin{remark}
In addition to the assumptions of Theorem~\ref{thm:wavelet}, if we assume that $U$ has $\alpha_1$ accuracy, $S$ has $\alpha_2$ accuracy, and $\alpha_3$ flatness, then by Lemma~\ref{lemma:accuracy} and Lemma~\ref{lemma:flatness}, $\GC[U]$ has at least $\alpha_1$ accuracy, $\GC[S]$ has at least $\min\{\alpha_2,\alpha_3\}$ accuracy, and at least $\alpha_3$ flatness. Combining these with Remark \ref{remark:vm}, we conclude that $\tau$ has at least $\min\{\alpha_1,\alpha_2,\alpha_3\}$ accuracy, and $t_\nu$ and $t_\nu^\du$, $\nu\in\set'$, have at least $\min\{\alpha_1,\alpha_2,\alpha_3\}$ vanishing moments.
\qquad\endproof
\end{remark}

\medskip
In order to prove Theorem~\ref{thm:wavelet}, we use the following lemma which connects the polyphase decomposition of the 1-D lowpass filter $H$ and the polyphase decomposition of the $\dm$-D prime coset sum lowpass filter $h$ obtained from $H$. Polyphase decomposition is a common method in Signal Processing and we give a brief review in Appendix~\ref{subS:ECLP}.

\begin{lemma}
\label{lemma:polyconnect}
Let $H$ be a 1-D lowpass filter with dilation $p$, and let $h$ be the $\dm$-D lowpass filter obtained from $H$ by applying the prime coset sum $\GC$. Let the sets $\set'$ and $\Fp'$, and the map $\eta:\set'\times \Fp'\to\set'$ be defined as before. Then for any $\nu\in\set'$, 
\beaN
(h(\nu+p\cdot))\widehat{\phantom{x}}(p\ome)={1\over (p-1)p^{\dm-1}}\sum_{l\in \Fp'}e^{i\ome\cdot (\nu-\eta(l,\nu)l)} \Big(H(l+p\cdot)\Big)\widehat{\phantom{x}}(p\ome\cdot\eta(l,\nu)), \ome\in\Td.
\eeaN
\end{lemma}

\medskip
\begin{proof}
First it is easy to see that (cf. (\ref{eq:polymask_s}) in Appendix~\ref{subS:ECLP})
$$
\hatH(\ome)\,=\,\sum_{l\in\Fp}e^{-i\ome l}\Big(H(l+p\cdot)\Big)\widehat{\phantom{x}}(p\ome),\quad\ome\in\TT.
$$
Using this identity and the definition of prime coset sum, we get
\bea
\hath(\ome)&\,=\,&{1\over (p-1)p^{\dm-1}}\left(1-p^{\dm-1}+\sum_{\nu\in\set'}\hatH(\ome\cdot\nu)\right),\quad \ome\in\Td \nonumber \\
&=&{1\over (p-1)p^{\dm-1}}\left(1-p^{\dm-1}+\sum_{\nu\in\set'}\sum_{l\in\Fp}e^{-i\ome\cdot\nu l}\Big(H(l+p\cdot)\Big)\widehat{\phantom{x}}(p\ome\cdot\nu)\right).
\label{eq:polydecom_h}
\eea
Next we use another identity that can be quickly derived (cf. (48) in \cite{Hur}):
\be
\label{eq:hhatpw}
(h(\nu+p\cdot))\widehat{\phantom{x}}(p\ome)={1\over p^\dm}\sum_{\gam\in\set^\ast}e^{i(\ome+\gam)\cdot\nu}\;\hath(\ome+\gam),\quad\ome\in\Td.
\ee
By using (\ref{eq:polydecom_h}), (\ref{eq:hhatpw}), and the fact that $\Big(H(l+p\cdot)\Big)\widehat{\phantom{x}}(p(\ome+\gam)\cdot\tilde{\nu})=\Big(H(l+p\cdot)\Big)\widehat{\phantom{x}}(p\ome\cdot\tilde{\nu})$, for any $l\in\Fp$, $\ome\in\Td$, $\gam\in\set^\ast$ and $\tilde{\nu}\in\set'$, we obtain $(h(\nu+p\cdot))\widehat{\phantom{x}}(p\ome)=$
\beaN
{1\over p^\dm}\sum_{\gam\in\set^\ast}e^{i(\ome+\gam)\cdot\nu}{1\over (p-1)p^{\dm-1}}\left(1-p^{\dm-1}+\sum_{\tilde{\nu}\in\set'}\sum_{l\in\Fp}e^{-i(\ome+\gam)\cdot\tilde{\nu} l}\Big(H(l+p\cdot)\Big)\widehat{\phantom{x}}(p\ome\cdot\tilde{\nu})\right).
\eeaN
Then we use the following simple identity (cf. (\ref{eq:dualchar})):
\beaN
\sum_{\gam\in\set^\ast}e^{i\gam\cdot\nu}=p^\dm\delta_{\nu,0}=
\cases{
      p^\dm,&\mbox{if $\nu=0$},\cr
      0,&\mbox{if $\nu\in\set'\bks0$},\cr
}
\eeaN
to get
\beaN
&&(h(\nu+p\cdot))\widehat{\phantom{x}}(p\ome) \\
&=&{1\over p^\dm}\sum_{\gam\in\set^\ast}e^{i(\ome+\gam)\cdot\nu}{1\over (p-1)p^{\dm-1}}\sum_{\tilde{\nu}\in\set'}\sum_{l\in\Fp'}e^{-i(\ome+\gam)\cdot\tilde{\nu} l}\Big(H(l+p\cdot)\Big)\widehat{\phantom{x}}(p\ome\cdot\tilde{\nu})\\
&=&{1\over p^\dm}{1\over (p-1)p^{\dm-1}}\sum_{\tilde{\nu}\in\set'}\sum_{l\in\Fp'}e^{i\ome\cdot(\nu-\tilde{\nu} l)}\Big(H(l+p\cdot)\Big)\widehat{\phantom{x}}(p\ome\cdot\tilde{\nu})\sum_{\gam\in\set^\ast}e^{i\gam\cdot(\nu-\tilde{\nu} l)},\quad\ome\in\Td.
\eeaN
Noting that $\sum_{\gam\in\set^\ast}e^{i\gam\cdot(\nu-\tilde{\nu} l)}=p^\dm$ if $\tilde{\nu}=\eta(l,\nu)$, and it is equal to $0$ otherwise, we obtain
\beaN
(h(\nu+p\cdot))\widehat{\phantom{x}}(p\ome)={1\over (p-1)p^{\dm-1}}\sum_{l\in \Fp'}e^{i\ome\cdot (\nu-\eta(l,\nu)l)} \Big(H(l+p\cdot)\Big)\widehat{\phantom{x}}(p\ome\cdot\eta(l,\nu)),\ome\in\Td,
\eeaN
as desired.
\end{proof}

\medskip
We now present the proof of Theorem~\ref{thm:wavelet}.

\begin{proof}[Proof of Theorem~\ref{thm:wavelet}]
Let $g$ and $h$ be the $\dm$-D lowpass filters associated with refinement masks $\GC[S]$ and $\GC[U]$.
Since $U$ is interpolatory, by Lemma~\ref{lemma:interpolatory}, $\GC[U]$ is also interpolatory, i.e., $h$ is interpolatory. Therefore, we can obtain the combined biorthogonal masks by using Result~\ref{result:result3}. By setting $\hatg:=\GC[S]$ and $\hath:=\GC[U]$ in Result~\ref{result:result3}, we obtain that, for every $\ome\in\Td$,
\beaN
\tau(\ome)&\,=\,&\hatg(\ome)+\Big(1-\sum_{\gam\in\set^\ast}\hatg(\ome+\gam)\overline{\hath(\ome+\gam)}\Big)\\
&=&\GC[S](\ome)+\left(1-\sum_{\gam\in\set^\ast}\GC[S](\ome+\gam)\overline{\GC[U](\ome+\gam)}\right),
\eeaN
and
\beaN
\tau^\du(\ome)=\hath(\ome)=\GC[U](\ome).
\eeaN

Since, in this case, $\dil=p{\tt I}_\dm$ and $q=p^\dm$, the $\dm$-D wavelet masks $t_\nu$, $\nu\in\set'$, are
\beaN
t_\nu(\ome)&\,=\,&e^{-i\ome\cdot\nu}-q\;\overline{(h(\nu+\dil\cdot))\widehat{\phantom{x}}(\dil^\ast\ome)}\\
&=&e^{-i\ome\cdot\nu}-p^\dm\;\overline{(h(\nu+p\cdot))\widehat{\phantom{x}}(p\ome)}, \quad \ome\in\Td.
\eeaN
Since $H$ is the $1$-D filter associated with $U$ and $h$ is the $\dm$-D filter associated with $\GC[U]$, by Lemma~\ref{lemma:polyconnect}, we have
\beaN
(h(\nu+p\cdot))\widehat{\phantom{x}}(p\ome)={1\over (p-1)p^{\dm-1}}\sum_{l\in \Fp'}e^{i\ome\cdot (\nu-\eta(l,\nu)l)} \Big(H(l+p\cdot)\Big)\widehat{\phantom{x}}(p\ome\cdot\eta(l,\nu)).
\eeaN
Therefore,
\beaN
t_\nu(\ome)&\,=\,&e^{-i\ome\cdot\nu}-p^\dm \; \overline{{1\over (p-1)p^{\dm-1}}\sum_{l\in \Fp'}e^{i\ome\cdot (\nu-\eta(l,\nu)l)} \Big(H(l+p\cdot)\Big)\widehat{\phantom{x}}(p\ome\cdot\eta(l,\nu))}
\\&=&
e^{-i\ome\cdot\nu}-{p\over p-1}\sum_{l\in \Fp'}e^{i\ome\cdot (\eta(l,\nu) l -\nu)}\; \overline{ \Big(H(l+p\cdot)\Big)\widehat{\phantom{x}}(p\ome\cdot\eta(l,\nu)) }\\
&=& e^{-i\ome\cdot\nu}\left(1-{p\over p-1}\sum_{l\in \Fp'}e^{i\ome\cdot \eta(l,\nu)l} \;\overline{U_l\Big(p\ome\cdot \eta(l,\nu)\Big)}\right),\quad\ome\in\Td.
\eeaN
The wavelet masks $t_\nu^\du$, $\nu\in\set'$, in (\ref{eq:tnudu}) can be obtained by applying similar arguments to the general form of $t_\nu^\du$, $\nu\in\set'$, in Result~\ref{result:result3}. This concludes that $(\tau,(t_\nu)_{\nu\in\set'})$ and $(\tau^\du,(t^\du_\nu)_{\nu\in\set'})$ defined as in Theorem~\ref{thm:wavelet} form $\dm$-D combined biorthogonal masks.
\end{proof}

\medskip
The following corollary of Theorem~\ref{thm:wavelet} may be useful on its own in some contexts. 

\begin{corollary}
\label{coro:biorthogonal}
Suppose that $S$ and $U$ are two 1-D refinement masks with prime dilation $p$, and that $U$ is interpolatory. Let $\GC$ be the prime coset sum. Then the two $\dm$-D refinement masks $\GC[U]$ and
$$\GC[S]+\left(1-\sum_{\gam\in\set^\ast}\GC[S](\cdot+\gam)\overline{\GC[U](\cdot+\gam)}\right)$$
with dilation $p{\tt I}_\dm$ are biorthogonal. 
\qquad\endproof
\end{corollary}

\begin{remark}
Of the two prime coset sum refinement masks $\GC[S]$ and $\GC[U]$, only the non-interpolatory mask $\GC[S]$ is modified by adding $1-\sum_{\gam\in\set^\ast}\GC[S](\cdot+\gam)\overline{\GC[U](\cdot+\gam)}$. We note that the statement of Corollary~\ref{coro:biorthogonal} holds true trivially for the case when $\GC[S]$ and $\GC[U]$ are already biorthogonal, since $1-\sum_{\gam\in\set^\ast}\GC[S](\cdot+\gam)\overline{\GC[U](\cdot+\gam)}= 0$ in such a case. One such case is when $S$ and $U$ are biorthogonal and $p=2$ (cf. Result~\ref{result:result1}(b)). 
\qquad\endproof
\end{remark}

\medskip
Next we illustrate our findings in two examples.

\begin{example}[\bf Centered $\dm$-D Haar combined biorthogonal masks with prime dilation $p$]
\label{example:4}
Let us consider the centered 1-D Haar refinement mask with prime dilation $p$.
We let 
\beaN
S(\ome)=U(\ome):=\disp{1\over p}\left(e^{i {{p-1}\over 2} \ome}+\cdots+e^{i\ome}+1+e^{-i\ome}+\cdots+e^{-i {{p-1}\over2} \ome} \right).
\eeaN
For example, when $p=3$, $S(\ome)=U(\ome):=\disp{1\over 3}\left(e^{i\ome}+1+e^{-i\ome}\right)$ as in Example \ref{example:2}.
Then they are both interpolatory with one accuracy. 
Now let us take $\set=\{-{{p-1}\over2},\cdots,-1,0,1,\cdots,{{p-1}\over2}\}^\dm$ and $\set^\ast={2\pi\over p}\{-{{p-1}\over2},\cdots,-1,0,1,\cdots,{{p-1}\over2}\}^\dm$ for any dimension $\dm\ge 2$. Then by Theorem~\ref{thm:wavelet} the $\dm$-D biorthogonal refinement masks 
\beaN\tau(\ome)=\tau^\du(\ome)=\disp{1\over p^\dm}\sum_{\nu\in\set}e^{-i\ome\cdot\nu},\quad \ome\in\Td,
\eeaN
and $\dm$-D wavelet masks
\beaN
t_\nu(\ome)=e^{-i\ome\cdot\nu}-1,\quad t^\du_{\nu}(\ome)=\disp{1\over p^\dm}\,e^{-i\ome\cdot\nu}-{1\over p^{2\dm}}\sum_{\mu\in\set}e^{-i\ome\cdot\mu},\quad \ome\in\Td,
\eeaN
for $\nu\in\set'$, form $\dm$-D combined biorthogonal masks. When $p=3$, the combined biorthogonal masks are studied in \cite{HZAsilomar}. By direct computation, we see that both $\tau$ and $\tau^\du$ have one accuracy, and that both $t_\nu$ and $t^\du_{\nu}$ have one vanishing moment for any $\nu\in\set'$. The number of nonzero entries, or the support of the filter associated with $t_\nu$ is only $2$ for any $\nu\in\set'$, and any dimension $\dm$ and dilation $p$.
\qquad\endproof
\end{example}

\begin{example}[\bf $2$-D combined biorthogonal masks with higher vanishing moments]
\label{example:5}
 Let $U$ be a $1$-D interpolatory refinement mask with dilation $3$ and accuracy $4$ \footnote{$U$ is obtained from 
\cite{KoSw}.}
 \begin{scriptsize}
\beaN
U(\ome):=\disp{1\over 3}\left(-{4\over81}e^{5i\ome}-{5\over81}e^{4i\ome}+{30\over81}e^{2i\ome}+{60\over81}e^{i\ome}+1+{60\over81}e^{-i\ome}+{30\over81}e^{-2i\ome}-{5\over81}e^{-4i\ome}-{4\over81}e^{-5i\ome}\right).
\eeaN
\end{scriptsize}
Let $S(\ome):=\disp{1\over 3}\left(e^{i\ome}+1+e^{-i\ome}\right)$. We take $\set=\{-1,0,1\}^2$ and $\set^\ast={2\pi\over3}\{-1,0,1\}^2$. Then by Theorem~\ref{thm:wavelet} the $2$-D biorthogonal refinement masks
\begin{scriptsize}
\beaN
\tau(\ome)={1\over9}\left({83\over27}+\sum_{\nu\in\set'}e^{-i\ome\cdot\nu}-{25\over81}\sum_{\nu\in\set'}e^{-3i\ome\cdot\nu}+{4\over81}\sum_{\nu\in\set'}e^{-6i\ome\cdot\nu}\right), \quad \ome\in\TT^2,
\eeaN
\beaN
\tau^\du(\ome)={1\over9}\left(1+{60\over81}\sum_{\nu\in\set'}e^{-i\ome\cdot\nu}+{30\over81}\sum_{\nu\in\set'}e^{-2i\ome\cdot\nu}-{5\over81}\sum_{\nu\in\set'}e^{-4i\ome\cdot\nu}-{4\over81}\sum_{\nu\in\set'}e^{-5i\ome\cdot\nu}\right),\quad \ome\in\TT^2,
\eeaN
\end{scriptsize}
and $2$-D wavelet masks
\beaN
t_{\nu}(\ome)=e^{-i\ome\cdot\nu}+{5\over81}e^{3i\ome\cdot\nu}-{60\over81}-{30\over81}e^{-3i\ome\cdot\nu}+{4\over81}e^{-6i\ome\cdot\nu},\quad \ome\in\TT^2,
\eeaN
\beaN
t_\nu^\du(\ome)={1\over9}\left(e^{-i\ome\cdot\nu}-\tau^\du(\ome)\right),\quad \ome\in\TT^2,
\eeaN
for $\nu\in\set'$, form $2$-D combined biorthogonal masks (cf. Figure \ref{figure:VM4} for the filters associated with $U$ and $\tau^\du$). Direct computation shows that $\tau$ has one accuracy, $\tau^\du$ has $4$ accuracy, $t_\nu$, $\nu\in\set'$, have $4$ vanishing moments, and $t_\nu^\du$, $\nu\in\set'$, have one vanishing moment. The support of the filter associated with $t_\nu$ is only $5$ for any $\nu\in\set'$.
\qquad\endproof
\end{example}

\begin{figure}[t]
\begin{picture}(500,270)
\put(68,250){
{
\centering
\small\begin{tabular}{ccccccccccc}
$-\frac{4}{81}$ & $-\frac{5}{81}$ & 0 & $\frac{30}{81}$ & $\frac{60}{81}$ & {\bf 1} &
$\frac{60}{81}$ & $\frac{30}{81}$ & 0 & $-\frac{5}{81}$ & $-\frac{4}{81}$\\[7pt]
\end{tabular}
}
}
\put(144,235){Filter associated with $U$}
\put(197,218){$\mathcal{C}_{2,3}$ (Prime coset sum)}
\put(192,230){\vector(0,-1){23}}
\put(61,100){
{
\centering
\small\begin{tabular}{ccccccccccc}
$-\frac{4}{81}$ & 0 & 0 & 0 & 0 & $-\frac{4}{81}$ & 0 & 0 & 0 & 0 & $-\frac{4}{81}$  \\[7pt]
0 & $-\frac{5}{81}$ & 0 & 0 & 0 & $-\frac{5}{81}$ & 0 & 0 & 0 & $-\frac{5}{81}$ & 0 \\[7pt]
0 & 0 & 0 & 0 & 0 & 0 & 0 & 0 & 0 & 0 & 0 \\[7pt] 
0 & 0 & 0& $\frac{30}{81}$ & 0& $\frac{30}{81}$  & 0 & $\frac{30}{81}$ & 0 &0 & 0\\[7pt]
0 & 0 & 0& 0& $\frac{60}{81}$& $\frac{60}{81}$ & $\frac{60}{81}$ & 0 &0 & 0 & 0\\[7pt]
$-\frac{4}{81}$ & $-\frac{5}{81}$ & 0 & $\frac{30}{81}$ & $\frac{60}{81}$ & {\bf 1} &
$\frac{60}{81}$ & $\frac{30}{81}$ & 0 & $-\frac{5}{81}$ & $-\frac{4}{81}$\\[7pt]
0 & 0 & 0& 0& $\frac{60}{81}$& $\frac{60}{81}$ & $\frac{60}{81}$ & 0 &0 & 0 & 0\\[7pt]
0 & 0 & 0& $\frac{30}{81}$ & 0& $\frac{30}{81}$  & 0 & $\frac{30}{81}$ & 0 &0 & 0\\[7pt]
0 & 0 & 0 & 0 & 0 & 0 & 0 & 0 & 0 & 0 & 0 \\[7pt] 
0 & $-\frac{5}{81}$ & 0 & 0 & 0 & $-\frac{5}{81}$ & 0 & 0 & 0 & $-\frac{5}{81}$ & 0 \\[7pt]
$-\frac{4}{81}$ & 0 & 0 & 0 & 0 & $-\frac{4}{81}$ & 0 & 0 & 0 & 0 & $-\frac{4}{81}$  \\[7pt]
\end{tabular}
}
}
\put(120,-3){Filter associated with $\tau^\du=\mathcal{C}_{2,3}[U]$}
\end{picture}
\caption{Lowpass filters associated with the masks $U$ and $\tau^\du$ in Example \ref{example:5}.}
\label{figure:VM4}
\end{figure}

\subsection{Algorithms}
\label{subS:algorithms}

Theorem~\ref{thm:wavelet} provides only one of many ways to obtain the non-redundant wavelet filter bank, given the two $\dm$-D refinement masks $\GC[S]$ and $\GC[U]$.
However, the resulting prime coset sum wavelet filter bank can be associated with fast algorithms that are faster than the usual tensor product ones. Below we present these fast prime coset sum algorithms.

\medskip\noindent{\bf Fast Prime Coset Sum Wavelet Algorithms.}
Let $G$ and $H$ be two 1-D lowpass filters with dilation $p$, where $H$ is interpolatory. In presenting our algorithms, we use the set $\Fp$ and the map $\eta$ that we defined in Section~\ref{subS:theory}. 

\medskip
{\obeylines{{\tt
{\it input}  $y_J: \Zd\to\RR$
\smallskip
{\bf (1) Decomposition Algorithm: computing $y_{j-1}$, $w_{\nu,j-1}$, $\nu\in\set'$ from $y_j$}
for $j=J,J-1,\ldots,1$
\ for $\nu\in\set'$ and $k\in\Zd$
\ \ $w_{\nu,j-1}(k)=y_j(pk+\nu)-\disp{1\over p-1}\disp\sum_{l\in\Fp'}\disp\sum_{m\equiv l}H(m)y_{j}(pk+\nu-\eta(l,\nu)m)$\hfill (i)
\ end
\ for $k\in\Zd$
\ \ $y_{j-1}(k)=y_j(pk)+\disp{1\over (p-1)p^{\dm}}\sum_{\nu\in\set'}\sum_{l\in\Fp'}\sum_{m\equiv l}G(m)w_{\nu,j-1}(k-\disp{ {\nu- \eta(l,\nu)}m \over p})$\hfill (ii)
\ end
end

\medskip
{\bf (2) Reconstruction Algorithm: computing $y_j$ from $y_{j-1}$, $w_{\nu,j-1}$, $\nu\in\set'$}
for $j=1,\ldots,J-1,J$ 
\ for $k\in\Zd$
\ \ $y_j(pk)=y_{j-1}(k)-\disp{1\over (p-1)p^{\dm}}\sum_{\nu\in\set'}\sum_{l\in\Fp'}\sum_{m\equiv l}G(m)w_{\nu,j-1}(k-\disp{\nu-{\eta(l,\nu)}m\over p})$ \hfill (iii)
\ end
\ for $\nu\in\set'$ and $k\in\Zd$
\  \ $y_j(pk+\nu)=w_{\nu,j-1}(k)+\disp{1\over p-1}\disp\sum_{l\in\Fp'}\disp\sum_{m\equiv l}H(m)y_{j}(pk+\nu-\eta(l,\nu)m)$ \hfill (iv)
\ end
end
}}}

\medskip
For decomposition, we compute the coarse coefficients $y_{j-1}$ and wavelet coefficients $w_{\nu,j-1}$, $\nu\in\set'$, from $y_j$. To obtain $w_{\nu,j-1}$, $\nu\in\set'$, we apply the filter associated with $t_\nu$, $\nu\in\set'$ to $y_j$, followed by downsampling with respect to the dilation matrix $\dil=p{\tt I}_\dm$, as is typically done in wavelet decomposition. Since $t_\nu$, $\nu\in\set'$, are written in terms of $U_l$, $l\in\Fp'$, and since $U_l$ can be written in terms of $1$-D filter $H$, we obtain the formula for Step {\tt (i)}. The proof of the identity in Step (i) is given in Appendix~\ref{subS:proofofalgorithm}, in which the concept of polyphase decomposition (cf. Appendix~\ref{subS:ECLP}) is used.

A key step of our decomposition algorithm is Step {\tt (ii)}. Typically, to obtain $y_{j-1}$, one needs to apply the filter associated with $\tau$ to $y_j$, followed by downsampling. However, since we have $\tau=\GC[S]+\left(1-\sum_{\gam\in\set^\ast}\GC[S](\cdot+\gam)\overline{\GC[U](\cdot+\gam)}\right)$ (cf. Theorem~\ref{thm:wavelet}) in this case, contrary to the filter associated with the first part of $\tau$, i.e. $\GC[S]$, it is not clear how the filter associated with the rest of the mask $\tau$, i.e. $1-\sum_{\gam\in\set^\ast}\GC[S](\cdot+\gam)\overline{\GC[U](\cdot+\gam)}$, would look like. As a result, the support of the filter associated with $\tau$ could be large. Therefore, the algorithm may not be necessarily faster than other wavelet algorithms if we use the filter associated with $\tau$ directly. However, by using the polyphase representation (cf. Appendix~\ref{subS:ECLP}), one can show that $y_{j-1}$ can also be derived by applying the filter associated with $\GC[S]$ (the first part of $\tau$) to $w_{\nu,j-1}$, $\nu\in\set'$. This is our Step {\tt (ii)}, and the details of exactly how it is done are written in Appendix~\ref{subS:proofofalgorithm}.

Our reconstruction algorithm is not the same as the typical wavelet reconstruction procedure either. We recall that the typical wavelet reconstruction is conducted by applying the reconstruction filters to $y_{j-1}$ and $w_{\nu,j-1}$, $\nu\in\set'$, upsampling them, and then summing them up. We reconstruct the signal by simply reversing Step {\tt (i)} and {\tt (ii)}. Step {\tt (iii)} is a reverse procedure of Step {\tt (ii)} that can always be performed. Step {\tt (iv)} is a reverse procedure of Step {\tt (i)}, and it is possible because the only $y_j$ needed in the right-hand side of Step {\tt (iv)} is $y_j(pk)$, which is already computed in Step {\tt (iii)}. 

\medskip\noindent
{\bf Complexity.}  
Next we discuss the complexity of the fast prime coset sum wavelet algorithms. We measure the complexity by counting the number of multiplicative operations needed in a complete cycle of $1$-level-down decomposition and $1$-level-up reconstruction, meaning the number of operations needed to fully derive $y_{j-1}$ and $w_{\nu,j-1}$, $\nu\in\set'$ from $y_j$, and to get back $y_j$. Here we only compute the number of multiplicative operations such as multiplication and division, as computing additive operations gives a similar result. 

Suppose that at level $j$, we have input data $y_j$ with $N$ data points. For simplicity, we assume that $N$ is a multiple of $p^\dm$, where $p$ is the dilation and $\dm$ is the spatial dimension. Then after $1$-level-down decomposition, we obtain $N/p^\dm$ coarse coefficients $y_{j-1}$ in Step {\tt (ii)}, and $N/p^\dm$ wavelet coefficients $w_{\nu,j-1}$ for each $\nu\in\set'$ in Step {\tt (i)}. We reconstruct the input data $y_j$ from coarse coefficients $y_{j-1}$ and wavelet coefficients $w_{\nu,j-1}$, $\nu\in\set'$. In particular, we obtain $N/p^\dm$ original data $y_j(pk)$ in Step {\tt (iii)} and $N/p^\dm$ original data $y_j(pk+\nu)$ for each $\nu\in\set'$ in Step {\tt (iv)}.

Suppose $\alpha$ and $\beta$ are the number of nonzero entries in the $1$-D lowpass filter $G$ and $H$, respectively. Recall that $H$ is interpolatory. 
Let
\beaN
\tilde{\alpha} \,:=\, \# \{ G(m) :\; G(m) \neq 0 \mbox{ and } m\equiv l \, (\mod p\ZZ)\mbox{ for some } l\in\Fp' \}.
\eeaN
Given the $N$ data points of the input data $y_j$, the number of multiplicative operations needed in a complete cycle of $1$-level-down decomposition and $1$-level-up reconstruction is the sum of

\begin{itemize}
\item $2\beta(p^\dm-1){N\over p^\dm}$ [for Step {\tt (i)} and {\tt (iv)}], and
\item $2\Big( (p^\dm-1)\tilde{\alpha}+n+1 \Big)  {N\over p^\dm}$ [for Step {\tt (ii)} and {\tt (iii)}]. 
\end{itemize}
Therefore, as a result, the complexity of the fast prime coset sum wavelet algorithms is
\be
\label{eq:complexity}
\left( { 2(p^\dm-1)\beta+2(p^\dm-1)\tilde{\alpha}+2n+2 \over p^\dm}\right)N.
\ee
Since $\tilde{\alpha}\le{p-1 \over p}(\alpha+1)$, this complexity is bounded above by
\beaN
 \Big(2\beta+2\,{p-1 \over p}(\alpha+1)+1 \Big)N.
\eeaN

Recall that in dyadic case, the fast tensor product wavelet algorithms have complexity $(\alpha+\beta)\dm N$, where $\alpha$ and $\beta$ are the number of nonzero entries of $1$-D lowpass filters, $\dm$ is the spatial dimension and $N$ is the data size (see, for example, \cite{HZ}). Therefore, the algorithm has linear complexity, i.e., $\sim CN$, with the data size $N$, where $C$ is some constant that does not depend on $N$. We refer to this constant as the {\it complexity constant}. The complexity constant for fast tensor product wavelet algorithm is $C_{TP}=(\alpha+\beta)\dm$. In particular, it grows linearly with the dimension $\dm$. Now let us consider the fast prime coset sum wavelet algorithm. In dyadic case, i.e., when $p=2$, the complexity is bounded above by $ (\alpha+2\beta+2)N$. Therefore, the complexity constant for the prime coset sum is $C_{PCS}=\alpha+2\beta+2$, which does not increase as dimension $\dm$ increases. Furthermore, since $\alpha\ge 2$, we have $C_{PCS}\le C_{TP}$ for all $\dm\ge 2$, which suggests that our fast prime coset sum algorithms can be much faster, at least in theory, than the fast tensor product algorithms when $\dm$ is large.

Our fast algorithms with $p=2$ are different from the original fast coset sum algorithms in \cite{HZ}, which results in a different complexity constant for the coset sum case. The complexity constant for the fast coset sum algorithms is $C_{CS}={3\over2}\alpha+2\beta$, and as a result, we have $C_{PCS}\le C_{CS}$ as long as $\alpha\ge 4$.

There are a couple of factors that contribute to make our algorithms this fast. Firstly, the number of nonzero entries in the $\dm$-D filter associated with $t_\nu$, $\nu\in\set'$, is essentially the same as that of the $1$-D filter $H$ (cf. Step ({\tt i})). Secondly, our decomposition algorithm is performed by bypassing the filter associated with $\tau$ (cf. Step ({\tt ii})), which could have large support, in general. Finally, the reconstruction algorithm has trivial reconstruction steps, which completely bypass the filters associated with $t_\nu^\du$, $\nu\in\set'$ (cf. Step ({\tt iii}) and ({\tt iv})).

We now discuss the fast algorithms for the prime coset sum wavelets in our previous examples.

\begin{example}({\bf Fast prime coset sum wavelet algorithms for the centered $\dm$-D Haar in Example~\ref{example:4}}).
\label{example:6}
Let us consider the centered $\dm$-D Haar combined biorthogonal masks with dilation $p$ constructed in Example~\ref{example:4}. For any fixed $p$, the $1$-D filter $G$ and $H$ are given as
\beaN
G(K)=H(K)=
\cases{
             1,\quad\mbox {if $K=0$},\cr
             1,\quad\mbox{if $K=\pm 1,\pm 2, \cdots, \pm {{p-1}\over 2}$},\cr
             0,\quad\mbox{otherwise}.\cr 
}
\eeaN
Then one can follow Step ({\tt i}) -- ({\tt iv}) with this pair of $G$ and $H$ to perform the fast algorithms. In this case, $\alpha=\beta=p$, $\tilde{\alpha}=p-1$. Hence for any dimension $\dm$, and input data of size $N$, the algorithms have complexity 
\beaN
\left( { 2p(p^\dm-1)+2(p-1)(p^\dm-1)+2n+2 \over p^\dm}\right)N
\le(4p-1)N.
\eeaN
Hence the complexity constant for a fixed $p$ is $4p-1$, and it is independent of the spatial dimension $\dm$.
\qquad\endproof
\end{example}

\begin{example}({\bf Fast prime coset sum wavelet algorithms for $2$-D wavelets with higher vanishing moments in Example~\ref{example:5}}).
\label{example:7}
Let us consider the $2$-D combined biorthogonal masks constructed in Example~\ref{example:5}. In this case, the $1$-D filter $G$ and $H$ are given as
\beaN
G(K)=
\cases{
             1,\quad\mbox {if $K=0$},\cr
             1,\quad\mbox{if $K=\pm 1$},\cr
             0,\quad\mbox{otherwise},\cr 
}      
\quad
H(K)=
\cases{
            1,&\mbox {if $K=0$},\cr
             {60\over81},&\mbox{if $K=\pm 1$},\cr
             {30\over81},&\mbox{if $K=\pm 2$},\cr
             -{5\over81},&\mbox{if $K=\pm 4$},\cr
             -{4\over81},&\mbox{if $K=\pm 5$},\cr
             0,&\mbox{otherwise}.\cr 
}
\eeaN
Then this pair of $G$ and $H$ can be used in Step ({\tt i}) -- ({\tt iv}) to implement the fast algorithms for the wavelet filter bank constructed in Example 5. In particular, since $\alpha=3$, $\beta=9$, $\tilde{\alpha}=2$, $p=3$ and $\dm=2$, the fast algorithms have complexity 
\beaN
\left( { 18(3^2-1)+4(3^2-1)+6 \over 3^2}\right)N
\le 21N,
\eeaN
for any input data of size $N$. Hence the complexity constant in this case is $21$.
\qquad\endproof
\end{example}

\section{Conclusion}
\label{section:conclusion}
In this paper we introduced a method called prime coset sum to construct multi-D refinement masks from $1$-D refinement masks. This method is a generalization of the existing method, the coset sum (\cite{HZ}), that works only for the dyadic dilations. We showed that for a prime dilation, the prime coset sum method maintains many important properties from the $1$-D refinement masks, such as interpolatory property, and under some conditions, the accuracy number. More importantly, the prime coset sum refinement masks can be used to construct wavelet filer banks with fast algorithms. Similar to the coset sum method for dyadic case, the prime coset sum fast algorithms have complexity constant that does not increase as the spatial dimension $\dm$ increases. This is contrary to the tensor product method, since its complexity constant increases linearly with the spatial dimension.

\appendix

\subsection{Proof of Lemma~\ref{lemma:interpolatory}}
\label{subS:proofoflemmainterpolatory}
Suppose $H$ and $h$ are the filters associated with masks $R$ and $\GC[R]$. If $R$ is interpolatory, by (\ref{eq:interpolatory}), $H(0)=1$, and $H(K)=0$ for any $K\in p\ZZ\bks 0$. Then, by (\ref{eq:filter}), $h(0)={1\over p-1}(p-p^\dm+(p^\dm-1)H(0))=1$, and $h(k)={1\over p-1}\sum_{l\in W_k}H(l)$ for any $k\ne0$. Since for each $k\in p\Zd\bks0$, every element $l$ in the set $W_k=\{l\in\ZZ\bks0: k=l\nu \hbox{ for some } \nu\in\set'\}$ must lie in $p\ZZ\bks0$, we see that $h(k)={1\over p-1}\sum_{l\in W_k}H(l)=0$ for any $k\in p\Zd\bks 0$. Hence $\GC[R]$ is interpolatory.

\subsection{Proof of Lemma~\ref{lemma:accuracy}}
\label{subS:proofoflemmaaccuracy}

First we note that $\GC[R]$ has at least accuracy number one, since $R$ has at least accuracy number one and $\GC$ is defined so that it preserves positive accuracy.

Let $\Fp^\ast$ be a complete set of representatives of the distinct cosets of $2\pi((p^{-1}\ZZ)/\ZZ)$ containing $0$. Since the order of zeros of $R$ at $\xi\in \Fp^\ast\bks0$ is $m_1$, and the order of zeros of $1-R$ at the origin is $m_2$, we have, for any integer $1\le k \le \min\{m_1,m_2\}-1$,
\begin{equation}
\label{eq:1Daccuracyresult}
(D^kR)(\xi)=0,\quad \mbox{for any }\xi\in \Fp^\ast.
\end{equation}
Thus, for any $\gam\in\set^\ast\bks0$ and any $\mu\in\Nd$ with $1\le |\mu| \le \min\{{m_1,m_2\}}-1$, where $|\mu|:=\mu_1+\cdots+\mu_\dm$, we get
\beaN
(D^\mu\GC[R])(\gam)&{\,=\,}&{1\over (p-1)p^{\dm-1}}\sum_{\nu\in\set'}(D^\mu[R(\ome\cdot\nu)])\mid_{\ome=\gam}\\
&=&{1\over (p-1)p^{\dm-1}}\sum_{\nu\in\set'}\left(\prod_{j=1}^\dm\nu_j^{\mu_j}\right)(D^{|\mu|}R)(\gam\cdot\nu)=0,
\eeaN
where the last equality is from (\ref{eq:1Daccuracyresult}) and the fact that $\gam\cdot\nu \,(\mod p\ZZ)$ belongs to $\Fp^\ast$. This implies the accuracy number of $\GC[R]$ is at least $\min\{m_1,m_2\}$.

\subsection{Review of Polyphase Representation of Wavelet Filter Banks}
\label{subS:ECLP}

The polyphase decomposition in \cite{V} is widely used in Signal Processing. We briefly review some relevant concepts in polyphase decomposition in terms of our notation and terminology, and refer other papers (e.g. \cite{DV2,Hur}) for details.

As before, we use $\dil$ to denote the dilation matrix, and $q$ to denote $|\det\dil|$. The polyphase decomposition transforms a filter (or signal) into $q$ filters (or signals) running at the sampling rate $1/q$. Let $\set$ be a complete set of representatives of the distinct cosets of $\Zd/\dil\Zd$ containing $0$, and let $\set'=\set\bks0$. For example, for the scalar dilation with $\lambda$, the set $\{0,1,\cdots,\lambda-1\}^\dm$ can be used for $\set$.  

The {\it polyphase decomposition of a synthesis filter} $h$ is defined as the Fourier series of $h(\nu+\dil\cdot)$, $\nu\in\set$:
\bea
\label{eq:poly_h}
{\tt H}_\nu(\ome):=(h(\nu+\dil\cdot))\widehat{\phantom{x}}(\ome)={\disp {1\over q}}\sum_{k\in\Zd}h(\nu+\dil k)e^{-ik\cdot \ome}, \quad \ome\in\Td,
\eea
and the {\it polyphase representation of a synthesis filter} $h$ is defined as the column $q$-vector of the form
\beaN
{\tt {H}}(\ome):=[{\tt H}_{\nu_0}(\ome),{\tt H}_{\nu_1}(\ome),\cdots,{\tt H}_{\nu_{q-1}}(\ome)]^T,\quad \ome\in\Td,
\eeaN
where $\nu_0=0$ and $\nu_j$, $j=1,\ldots,q-1$, are the ordered elements of the set $\set'$.
Then it is easy to see that the Fourier series of $h$ can be written in terms of the polyphase decomposition of $h$ as follows:
\be
\label{eq:polymask_s}
\hath(\ome)=\sum_{\nu\in\set}e^{-i\ome\cdot\nu}{\tt H}_\nu(\dil^\ast\ome).
\ee
Similarly, the {\it polyphase decomposition of an analysis filter} $g$ is defined as the complex conjugate of the Fourier series of $g(\nu+\dil\cdot)$, $\nu\in\set$:
\bea
\label{eq:poly_g}
{\tt G}_\nu(\ome):=\overline{(g(\nu+\dil\cdot))\widehat{\phantom{x}}(\ome)}={\disp {1\over q}}\sum_{k\in\Zd}g(\nu-\dil k)e^{-ik\cdot \ome}, \quad \ome\in\Td,
\eea
and the {\it polyphase representation of an analysis filter} $g$ is defined as the row $q$-vector of the form 
\beaN
{\tt {G}}(\ome):=[{\tt G}_{\nu_0}(\ome),{\tt G}_{\nu_1}(\ome),\cdots,{\tt G}_{\nu_{q-1}}(\ome)],\quad \ome\in\Td,
\eeaN
and, as a result, we have the identity
\beaN
\overline{\hatg(\ome)}=\sum_{\nu\in\set}e^{i\ome\cdot\nu}{\tt G}_\nu(\dil^\ast\ome).
\eeaN
Under these notations, it is easy to see that $h$ and $g$ are biorthogonal if and only if ${\tt G}(\ome){\tt H}(\ome)={1/q}$.

A filter bank (that is non-redunant with perfect reconstruction property) can be represented by two $q\times q$ polyphase matrices ${\tt A}(\ome)$ and ${\tt S}(\ome)$ that satisfy ${\tt S}(\ome){\tt A}(\ome)=(1/q){\tt I}_{q}$. 
The row vectors of ${\tt A}(\ome)$ represent the polyphase representation of analysis filters, where the first row corresponding to the lowpass filter and the rest to the highpass filters. The column vectors of ${\tt S}(\ome)$ represent the polyphase representation of synthesis filters, where the first column corresponding to the lowpass filter and the rest to the highpass filters.

We finish this subsection by stating Result~\ref{result:result3} in terms of the polyphase representation, as it will be useful in the later part of the paper.

\begin{result}[\bf Result~\ref{result:result3} stated in terms of polyphase representation]
\label{result:result4}
Suppose $g$ and $h$ are two $\dm$-D lowpass filters with dilation $\dil$, and $h$ is interpolatory. Let ${\tt G}(\ome)$ and ${\tt H}(\ome)$ be the polyphase representation of $g$ and $h$ with length $q=|\det\dil|$, and let $\tilde{\tt{G}}(\ome)$ and $\tilde{\tt{H}}(\ome)$ be the subvectors of $\tt{G}(\ome)$ and $\tt{H}(\ome)$ of length $q-1$, respectively, obtained by removing the first entry. Then the following two polyphase matrices
\bea
\label{eq:fb}
{\tt A}(\ome):=\left[\begin{array}{cc}{\tt G}_{\nu_0}(\ome)+q\; {\tt B}(\ome)&\quad \tilde{\tt G}(\ome)\\ -q \;\tilde{\tt H}(\ome)&\quad {\tt I}_{q-1}\end{array}\right]
\quad
{\tt S}(\ome):=\left[\begin{array}{cc}\disp{1\over q}&\quad \disp{1\over q}\;\tilde{\tt G}(\ome)\\ \tilde{\tt H}(\ome)&\quad \disp{1\over q}{\tt I}_{q-1}-\tilde{\tt H}(\ome)\tilde{\tt G}(\ome)\end{array}\right]
\eea
satisfy ${\tt S}(\ome){\tt A}(\ome)=(1/ q){\tt I}_{q}$, where ${\tt B}(\ome):=1/q-{\tt G}(\ome){\tt H}(\ome)$.
\qquad\endproof
\end{result}

\subsection{Proof of Result~\ref{result:result3}}
\label{subS:proofofresult3}
We want to show that $\tau$, $\tau^\du$, $t_\nu$ and $t^\du_{\nu}$, $\nu\in\set'$, in Result~\ref{result:result3}, satisfy the following identity (cf. (\ref{eq:muep}) in Section~\ref{subS:notation})
\beaN
\overline{\tau(\ome+\gam)}\tau^\du(\ome)+\sum_{\nu\in\set'}\overline{t_\nu(\ome+\gam)}t_\nu^\du(\ome)=\delta_{\gam,0}=
\cases{
1,\quad\mbox{if $\gam=0$},\cr
           0,\quad\mbox{if $\gam\in\set^\ast\bks0$}.\cr
}
\eeaN
By substituting the masks $\tau$, $\tau^\du$, $t_\nu$ and $t^\du_\nu$, $\nu\in\set'$, in Result~\ref{result:result3}, we get
\beaN
&& \overline{\tau(\ome+\gam)}\tau^\du(\ome)+\sum_{\nu\in\set'}\overline{t_\nu(\ome+\gam)}t_\nu^\du(\ome)\\
&\,=\,& \Big(\overline{\hatg(\ome+\gam)}+\Big(1-\sum_{\tilde{\gam}\in\set^\ast}\overline{\hatg(\ome+\tilde{\gam}+\gam)}\;\hath(\ome+\tilde{\gam}+\gam)\Big)\Big)\hath(\ome)  \\
&& + \sum_{\nu\in\set'}\left(e^{i(\ome+\gam)\cdot\nu}-q(h(\nu+\dil\cdot))\widehat{\phantom{x}}(\dil^\ast\ome)\right)\left({1\over q}e^{-i\ome\cdot\nu}-\overline{(g(\nu+\dil\cdot))\widehat{\phantom{x}}(\dil^\ast\ome)} \; \hath(\ome)\right)\\
&\,=\,&  \overline{\hatg(\ome+\gam)}\;\hath(\ome)+\hath(\ome)-\sum_{\tilde{\gam}\in\set^\ast}\overline{\hatg(\ome+\tilde{\gam}+\gam)}\;\hath(\ome+\tilde{\gam}+\gam)\hath(\ome)  \\
&& + {1\over q}\sum_{\nu\in\set}e^{i\gam\cdot\nu}-{1\over q}-\left(\sum_{\nu\in\set}e^{i(\ome+\gam)\cdot\nu} \; \overline{(g(\nu+\dil\cdot))\widehat{\phantom{x}}(\dil^\ast\ome)} \; \hath(\ome)-\overline{g(\dil\cdot)\widehat{\phantom{x}}(\dil^\ast\ome)}\; \hath(\ome)\right)  \\
&& - \left(\sum_{\nu\in\set}e^{-i\ome\cdot\nu}(h(\nu+\dil\cdot))\widehat{\phantom{x}}(\dil^\ast\ome)-h(\dil\cdot)\widehat{\phantom{x}}(\dil^\ast\ome)\right)  \\
&& + 
q\sum_{\nu\in\set'}(h(\nu+\dil\cdot))\widehat{\phantom{x}}(\dil^\ast\ome) \; \overline{(g(\nu+\dil\cdot))\widehat{\phantom{x}}(\dil^\ast\ome)} \; \hath(\ome).
\eeaN
It is easy to see that the following identity is true:
\be
\label{eq:char}
\sum_{\nu\in\set}e^{i\gam\cdot\nu}=q\delta_{\gam,0}=
\cases{
q,\quad\mbox{if $\gam=0$},\cr
             0,\quad\mbox{if $\gam\in\set^\ast\bks0$},\cr
}
\ee
where $q=|\det\dil|$. Since $h$ is interpolatory, we have
\be
\label{eq:zerocoset}
h(\dil\cdot)\widehat{\phantom{x}}(\dil^\ast\ome)={1\over q} \; .
\ee
Then by using (\ref{eq:char}), (\ref{eq:zerocoset}), (\ref{eq:polymask_s}), and the fact that $(g(\nu+\dil\cdot))\widehat{\phantom{x}}(\dil^\ast\ome)=(g(\nu+\dil\cdot))\widehat{\phantom{x}}(\dil^\ast(\ome+\gam))$, for any $\nu\in\set$, $\ome\in\Td$, and $\gam\in\set^\ast$, we get
\beaN
&&\overline{\tau(\ome+\gam)}\tau^\du(\ome)+\sum_{\nu\in\set'}\overline{t_\nu(\ome+\gam)}t_\nu^\du(\ome)\\
&\,=\,& \delta_{\gam,0}-\sum_{\tilde{\gam}\in\set^\ast}\overline{\hatg(\ome+\tilde{\gam})}\;\hath(\ome+\tilde{\gam})\hath(\ome)+q\sum_{\nu\in\set}(h(\nu+\dil\cdot))\widehat{\phantom{x}}(\dil^\ast\ome) \; \overline{(g(\nu+\dil\cdot))\widehat{\phantom{x}}(\dil^\ast\ome)} \; \hath(\ome) \\
&=& \delta_{\gam,0}-\left(\sum_{\tilde{\gam}\in\set^\ast}\overline{\hatg(\ome+\tilde{\gam})}\;\hath(\ome+\tilde{\gam})-q\sum_{\nu\in\set}(h(\nu+\dil\cdot))\widehat{\phantom{x}}(\dil^\ast\ome) \; \overline{(g(\nu+\dil\cdot))\widehat{\phantom{x}}(\dil^\ast\ome)} \right) \hath(\ome).
\eeaN
Moreover, by (\ref{eq:polymask_s}), and the dual identity of (\ref{eq:char}):
\bea
\label{eq:dualchar}
\sum_{\gam\in\set^\ast}e^{i\gam\cdot\nu}=q\delta_{\nu,0}=
\cases{
q,\quad\mbox{if $\nu=0$},\cr
              0,\quad\mbox{if $\nu\in\set'\bks0$},\cr
}
\eea
we have
\beaN
&&\sum_{\gam\in\set^\ast}\overline{\hatg(\ome+\gam)}\;\hath(\ome+\gam)\\
&\,=\,&\sum_{\gam\in\set^\ast}\left(\sum_{\nu\in\set}e^{i(\ome+\gam)\cdot\nu}\; \overline{(g(\nu+\dil\cdot))\widehat{\phantom{x}}(\dil^\ast\ome)}\right)\left(\sum_{\tilde{\nu}\in\set}e^{-i(\ome+\gam)\cdot\tilde{\nu}}\; (h(\tilde{\nu}+\dil\cdot))\widehat{\phantom{x}}(\dil^\ast\ome)\right)\\
&=&\sum_{\tilde{\nu}\in\set}\left(\sum_{\nu\in\set}\Big(\sum_{\gam\in\set^\ast}e^{i\gam\cdot(\nu-\tilde{\nu})}\Big)e^{i\ome\cdot\nu} \; \overline{(g(\nu+\dil\cdot))\widehat{\phantom{x}}(\dil^\ast\ome)} \right)e^{-i\ome\cdot\tilde{\nu}}(h(\tilde{\nu}+\dil\cdot))\widehat{\phantom{x}}(\dil^\ast\ome)\\
&=& q\sum_{\nu\in\set}(h(\nu+\dil\cdot))\widehat{\phantom{x}}(\dil^\ast\ome) \; \overline{(g(\nu+\dil\cdot))\widehat{\phantom{x}}(\dil^\ast\ome)}.
\eeaN
Therefore,
\beaN
\overline{\tau(\ome+\gam)}\tau^\du(\ome)+\sum_{\nu\in\set'}\overline{t_\nu(\ome+\gam)}t_\nu^\du(\ome)=\delta_{\gam,0}.
\eeaN
This concludes the proof.

\subsection{Proof of the identities in the decomposition algorithm}
\label{subS:proofofalgorithm}
The {\it polyphase decomposition of a signal} $y_j$ with respect to the dilation matrix $\dil=p{\tt I}_\dm$, with $q=|\det\dil|=p^\dm$, is defined as the Fourier series of $y_j(\nu+p\cdot)$, $\nu\in\set$:
\beaN
{\tt Y}_{\nu,j}(\ome):=(y_j(\nu+p\cdot))\widehat{\phantom{x}}(\ome)={\disp {1\over q}}\sum_{k\in\Zd}y_j(\nu+p k)e^{-ik\cdot \ome}, \quad \ome\in\Td,
\eeaN
and the {\it polyphase representation of a signal} $y_j$ is defined as the column $q$-vector of the form
\beaN{\tt {Y_j}}(\ome):=[{\tt Y}_{\nu_0,j}(\ome),{\tt Y}_{\nu_1,j}(\ome),\cdots,{\tt Y}_{\nu_{q-1},j}(\ome)]^T,\quad \ome\in\Td,
\eeaN
where $\nu_0=0$ and $\nu_j$, $j=1,\ldots,q-1$, are the ordered elements of the set $\set'$.
Let $Y_{j-1}$ and $W_{\nu,j-1}$ be the Fourier series of coarse coefficients $y_{j-1}$ and wavelet coefficients $w_{\nu,j-1}$, $\nu\in\set'$, respectively, 
\beaN
Y_{j-1}(\ome)&{\,:=\,}&{1\over q}\sum_{k\in\Zd}y_{j-1}(k)e^{-ik\cdot\ome},\\
W_{\nu,j-1}(\ome)&:=&{1\over q}\sum_{k\in\Zd}w_{\nu,j-1}(k)e^{-ik\cdot\ome}, \quad \nu\in\set',
\eeaN
for every $\ome\in\Td$.
Then a $1$-level-down decomposition, in frequency domain, can be written as
\beaN
\left[\begin{array}{c}Y_{j-1}(\ome)\\ W_{j-1}(\ome)\end{array}\right]={\tt A}(\ome)\left[\begin{array}{c}{\tt Y}_{\nu_0,j}(\ome)\\ \tilde{\tt Y}_{j}(\ome)\end{array}\right],
\eeaN
where $ W_{j-1}(\ome):=[W_{{\nu_1},j-1}(\ome),\cdots,W_{{\nu_{q-1}},j-1}(\ome)]^T$ and $\tilde{\tt Y}_j(\ome)$ is a subvector of ${\tt Y}(\ome)$ of length $q-1$ obtained by removing the first entry. 

A key observation, which is also part of the reason why the fast prime coset sum wavelet algorithms is fast, is that ${\tt A}(\ome)$ as defined in (\ref{eq:fb}) can be decomposed into two triangular matrices:
$$
{\tt A}(\ome)=
\left[\begin{array}{cc}1&~~~~~~~\tilde{\tt G}(\ome)\\0&~~~~~~~{\tt I}_{q-1}\end{array}\right]
\left[\begin{array}{cc}1&0\\ -q\;\tilde{\tt H}(\ome)~~~~~&{\tt I}_{q-1}\end{array}\right].
$$
Thus we can calculate $W_{j-1}(\ome)$ first, then use $ W_{j-1}(\ome)$ to compute $Y_{j-1}(\ome)$ as follows,
\bea
\label{eq:i}
W_{j-1}(\ome)&{\,=\,}&\,-q\;\tilde{\tt H}(\ome){\tt Y}_{\nu_0,j}(\ome)+\tilde{\tt Y}_{j}(\ome),\\
\label{eq:ii}
Y_{j-1}(\ome)&=&{\tt Y}_{\nu_0,j}(\ome)+\tilde{\tt G}(\ome)W_{j-1}(\ome).
\eea
From these (\ref{eq:i}) and (\ref{eq:ii}), we now derive Step {\tt (i)} and {\tt (ii)} in our decomposition algorithm. 

From (\ref{eq:i}), (\ref{eq:poly_h}) and Lemma~\ref{lemma:polyconnect}, we know that, for any $\nu\in\set'$, 
\beaN
W_{\nu,j-1}(\ome)&{\,=\,}&\,- q\;{\tt H}_\nu(\ome) {\tt Y}_{\nu_0,j}(\ome)+{\tt Y}_{\nu,j}(\ome)\\
&=&\,-q\;(h(\nu+p\cdot))\widehat{\phantom{x}}(\ome) {\tt Y}_{\nu_0,j}(\ome)+{\tt Y}_{\nu,j}(\ome) \\
&=&\,-{p\over p-1}\sum_{l\in\Fp'}e^{i\ome\cdot{ (\nu-\eta(l,\nu) l) \over p} }\Big(H(l+p\cdot)\Big)\widehat{\phantom{x}}(\ome\cdot\eta(l,\nu)){\tt Y}_{\nu_0,j}(\ome)+{\tt Y}_{\nu,j}(\ome).
\eeaN
Hence,
\beaN
\disp{1\over p^\dm}\sum_{k\in\Zd}&&w_{\nu,j-1}(k)e^{-i k\cdot\ome}=W_{\nu,j-1}(\ome)
={1\over p^\dm}\sum_{k\in\Zd}y_j(pk+\nu)e^{-ik\cdot\ome}\\
&&\,-{p\over p-1}\sum_{l\in\Fp'}e^{i\ome\cdot{ (\nu-\eta(l,\nu) l) \over p} }\disp{1\over p}\sum_{m\in\ZZ}H(l+pm)e^{-im(\ome\cdot\eta(l,\nu)) }{1\over p^\dm}\sum_{k'\in\Zd}y_j(pk')e^{-ik'\cdot\ome}.
\eeaN
Therefore,
\beaN
&&\sum_{k\in\Zd}w_{\nu,j-1}(k)e^{-ik\cdot\ome}\\
&{\,=\,}& \,\sum_{k\in\Zd}y_j(pk+\nu)e^{-ik\cdot\ome} \\
&& ~~~~ -{1\over p-1}\sum_{l\in\Fp'}\sum_{m\in\ZZ}\sum_{k'\in\Zd}e^{i\ome\cdot{ (\nu-\eta(l,\nu) l) \over p} }e^{-im(\ome\cdot\eta(l,\nu)) }e^{-ik'\cdot\ome}H(l+pm)y_j(pk') \\
&=&\,\sum_{k\in\Zd}y_j(pk+\nu)e^{-ik\cdot\ome} \\
&& ~~~~ -{1\over p-1}\sum_{l\in\Fp'}\sum_{m\in\ZZ}\sum_{k\in\Zd}e^{-ik\cdot\ome}H(l+pm)y_j(pk+\nu-\eta(l,\nu)(pm+l))\\
&=&\sum_{k\in\Zd}\left(y_j(pk+\nu)-\disp{1\over p-1}\disp\sum_{l\in\Fp'}\disp\sum_{m\in\ZZ}H(l+pm)y_{j}(pk+\nu-\eta(l,\nu)(pm+l))\right)e^{-ik\cdot\ome}\\
&=&\sum_{k\in\Zd}\left(y_j(pk+\nu)-\disp{1\over p-1}\disp\sum_{l\in\Fp'}\disp\sum_{m\equiv l}H(m)y_{j}(pk+\nu-\eta(l,\nu)(m))\right)e^{-ik\cdot\ome},
\eeaN
which in turn implies that we have for any $k\in\Zd$ and $\nu\in\set'$,
\beaN
w_{\nu,j-1}(k)\,=\,y_j(pk+\nu)-\disp{1\over p-1}\disp\sum_{l\in\Fp'}\disp\sum_{m\equiv l}H(m)y_{j}(pk+\nu-\eta(l,\nu)m).
\eeaN
This is exactly Step {\tt (i)} in our decomposition algorithm.

From (\ref{eq:ii}), (\ref{eq:poly_g}) and by Lemma~\ref{lemma:polyconnect} we know that
\beaN
&&Y_{j-1}(\ome){\,=\,}{\tt Y}_{\nu_0,j}(\ome)+\sum_{\nu\in\set'}{\tt G}_\nu(\ome)W_{\nu,j-1}(\ome)\\
&&={\tt Y}_{\nu_0,j}(\ome)+\sum_{\nu\in\set'}\overline{g(\nu+p\cdot)\widehat{\phantom{x}}(\ome)}\; W_{\nu,j-1}(\ome)\\
&&={\tt Y}_{\nu_0,j}(\ome)+\sum_{\nu\in\set'}{1\over (p-1)p^{\dm-1}}\sum_{l\in\Fp'}e^{i\ome\cdot{ (\eta(l,\nu) l-\nu) \over p} } \; \overline{\Big(G(l+p\cdot)\Big)\widehat{\phantom{x}}(\ome\cdot\eta(l,\nu))} \; W_{\nu,j-1}(\ome).
\eeaN
Hence,
\beaN
&&{1\over p^\dm}\sum_{k\in\Zd}y_{j-1}(k)e^{-ik\cdot\ome}=Y_{j-1}(\ome)={1\over p^\dm}\sum_{k\in\Zd}y_j(pk)e^{-ik\cdot\ome}+\\
&&\sum_{\nu\in\set'}{1\over (p-1)p^\dm}\sum_{l\in\Fp'}e^{i\ome\cdot{ (\eta(l,\nu) l-\nu) \over p} }   \sum_{m\in\Zd}G(l-pm)e^{-im(\ome\cdot\eta(l,\nu))} {1\over p^\dm}\sum_{k'\in\Zd}w_{\nu,j-1}(k')e^{-ik'\cdot\ome}.
\eeaN
Therefore, we have 
\beaN
&&\sum_{k\in\Zd}y_{j-1}(k)e^{-ik\cdot\ome}\\
&{\,=\,}&\sum_{k\in\Zd}y_j(pk)e^{-ik\cdot\ome} \\
&& +{1\over (p-1)p^{\dm}}\sum_{\nu\in\set'}\sum_{l\in\Fp'}\sum_{m\in\Zd}\sum_{k'\in\Zd}e^{i\ome\cdot{ (\eta(l,\nu) l-\nu) \over p} }e^{-im(\ome\cdot\eta(l,\nu))}e^{-ik'\cdot\ome}G(l-pm)w_{\nu,j-1}(k') \\
&{\,=\,}&\sum_{k\in\Zd}y_j(pk)e^{-ik\cdot\ome} \\
&& +{1\over (p-1)p^{\dm}}\sum_{\nu\in\set'}\sum_{l\in\Fp'}\sum_{m\in\Zd}\sum_{k\in\Zd}e^{-ik\cdot\ome}G(l-pm)w_{\nu,j-1}(k-{\nu-\eta(l,\nu) l \over p} - \eta(l,\nu)m)\\
&{\,=\,}&\sum_{k\in\Zd}\left( y_j(pk)+{1\over (p-1)p^{\dm}}\sum_{\nu\in\set'}\sum_{l\in\Fp'}\sum_{m\equiv l}G(m)w_{\nu,j-1}(k-{\nu-\eta(l,\nu) m \over p} )\right)e^{-ik\cdot\ome},
\eeaN
As a result, we have, for any $k\in\Zd$,
\beaN
y_{j-1}(k){\,=\,}
y_j(pk)+{1\over (p-1)p^{\dm}}\sum_{\nu\in\set'}\sum_{l\in\Fp'}\sum_{m\equiv l}G(m)w_{\nu,j-1}(k-{\nu-\eta(l,\nu)m\over p}).
\eeaN
This is exactly Step {\tt (ii)} in our decomposition algorithm.

\bibliographystyle{plain}
\bibliography{IEEEabrv,draft}







\end{document}